\documentclass[11pt]{amsart}
\usepackage{latexsym}
\usepackage{amsmath,amsthm,amssymb, amscd,color, mathtools,bbm}
\usepackage[bb=libus]{mathalpha}
\usepackage{mathrsfs}
\usepackage[all]{xy}
\usepackage[all]{xy}
\usepackage{tikz}
\usepackage{tikz-cd}
\usepackage{adjustbox}
\usetikzlibrary{decorations.pathmorphing}

    % ...
\usepackage{blkarray} 
\usepackage{framed}
\theoremstyle{plain}
\newtheorem{theorem}{Theorem}[section]
\newtheorem{lemma}[theorem]{Lemma}
\newtheorem{corollary}[theorem]{Corollary}
\numberwithin{equation}{section}

\newtheorem{innercustomgeneric}{\customgenericname}
\providecommand{\customgenericname}{}
\newcommand{\newcustomtheorem}[2]{%
	\newenvironment{#1}[1]
	{%
		\renewcommand\customgenericname{#2}%
		\renewcommand\theinnercustomgeneric{##1}%
		\innercustomgeneric
	}
	{\endinnercustomgeneric}
}

\newcustomtheorem{customthm}{Theorem}

\theoremstyle{definition}
\newtheorem{definition}[theorem]{Definition}
\newtheorem{example}[theorem]{Example}

\newtheorem{proposition}[theorem]{Proposition}
\newtheorem{remark}[theorem]{Remark}

\theoremstyle{remark}

\newcommand{\CC}{\mathbb{C}}

\newcommand{\ZZ}{\mathbb{Z}}

\newcommand{\WG}{\mbox{WGr}}
\newcommand{\G}{\mbox{Gr}}
\newcommand{\WW}{\mathbb{W}}
\newcommand{\Pl}{Pl\"{u}cker}

\usepackage{blkarray} 
\usepackage{framed}

\newcustomtheorem{customlemma}{Lemma}
\newcustomtheorem{customprop}{Proposition}
\newcustomtheorem{customcorollary}{Corollary}

\begin{document}

\title[$K$-theory of divisive weighted Grassmann orbifold] {Twisted factorial Grothendieck polynomials and equivariant $K$-theory of weighted Grassmann orbifolds}

\author[K Brahma]{Koushik Brahma}
\address{Department of Pure and Applied Mathematics, Waseda University, 3-4-1 Okubo, Shinjuku-ku, Tokyo 169-8555, Japan.}
\email{w.iac24166@kurenai.waseda.jp koushikbrahma95@gmail.com}



\date{\today}
\dedicatory{}
\abstract In this paper, we provide an explicit description of the Schubert classes in the equivariant $K$-theory of weighted Grassmann orbifolds. We introduce the `twisted factorial Grothendieck polynomials', a family of symmetric polynomials by specializing the factorial Grothendieck polynomials, and prove that they represent the Schubert classes in the equivariant $K$-theory of the weighted Grassmann orbifolds. We give an explicit formula for the restriction of the Schubert classes to any torus fixed point in terms of twisted factorial Grothendieck polynomials. We give an explicit formula for the structure constants with respect to the Schubert basis in the equivariant $K$-theory of weighted Grassmann orbifolds. Eminently, we describe `twisted Grothendieck polynomials' and prove that these represent the Schubert classes in the $K$-theory of the weighted Grassmann orbifold. As a consequence, we describe the structure constants in the $K$-theory of weighted Grassmann orbifolds.
\endabstract

\maketitle

\tableofcontents 

\section{Introduction}

The primary objective of Schubert calculus is to compute the structure constants of the cohomology ring of the (partial) flag variety, with respect to the basis formed by Schubert classes; see \cite{km,Man,Fu} and reference therein for historical background and foundational developments. One effective approach to determine these structure constants is to realize Schubert classes as explicit symmetric polynomials. For instance, in the case of Grassmannians, Schubert classes in the equivariant cohomology are represented by factorial Schur polynomials \cite{MoSe}. Other notable examples include (double or, quantum) Schubert polynomials \cite{LaSc,FGP}, which correspond to the Schubert classes in (equivariant or, quantum) cohomology of full flag varieties, and (factorial) Schur $Q$-polynomials \cite{Ik,Iva}, which arise in the study of the (equivariant) cohomology of Lagrangian Grassmannians. These polynomial realizations offer a significant advantage: they enable the study of structure constants through the explicit multiplication of polynomials. In \cite{MS}, Molev-Segan compute the Littlewood-Richardson rule in the equivariant cohomology of Grassmannian in terms of factorial Schur polynomials.

In the same direction, Lascoux-Schützenberger introduced the (double) Grothendieck polynomials in \cite{LS} as representatives for (equivariant) $K$-theory classes determined by Schubert structure sheaves of flag varieties.
A Littlewood–Richardson rule for the $K$-theory of Grassmannians was discovered by Buch \cite{Buch},  who provided the initial rule governing the multiplication of Schubert structure sheaves in the $K$-theory.  Since then, their properties were studied by Lenart \cite{Len}, Ikeda-Shimazaki \cite{IkeShi}, among others. McNamara introduced the factorial Grothendieck polynomial in \cite{mc} corresponding to the Grassmannian permutations. The Grassmannian permutations  are naturally indexed by partitions, where Grothendieck polynomials acquire symmetry in their variables. Throughout the article, we use the terminology (factorial) Grothendieck polynomials, with the implicit understanding that this designation always pertains to Grassmannian permutations, or, equivalently, to the (equivariant) $K$-theory of the Grassmannians. Pechenik-Yong found combinatorial rule of the structure constants with respect to Schubert class in the equivariant $K$-theory of the Grassmannians in \cite{PY,PY1}.

%We focus on the case of Grassmannian permutations when the Grothendieck polynomials become symmetric in their variables and are more naturally indexed by partitions. We will continue to refer to them simply as Grothendieck polynomials,with the implicit understanding that this always refers to Grassmannian permutations, or equivalently to the K-theory of the Grassmannian.
%Let $K_T(X)$ denote the Grothendieck ring of $ T$-equivariant vector bundles over $X$. This ring has a natural $K_T(pt)$-module structure and an additive basis given by the classes of Schubert structure sheaves; for background, we refer the reader to, e.g., [KoKu90, AnGrMi11] and the references therein. 

%The analogs of Littlewood-Richardson coefficients in the equivariant $K$-theory of Grassmannians are the Laurent polynomials. 

%In \cite{Ka}, Kawasaki introduced a ${\bf{b}}$-action of $\CC^*$ on $\CC^{m+1}\setminus\{\mathbf{0}\}$ for a weight vector ${\bf{b}}=(b_0,\dots,b_m)\in(\ZZ_{\geq 1})^{m+1}$ as follows:
%\begin{equation}\label{eq_weigh_act}	t(z_0,z_1,z_2,\dots,z_m)=(t^{b_0}z_0,t^{b_1}z_1,t^{b_2}z_2,\dots,t^{b_m}z_m).
%\end{equation}

Kawasaki introduced the weighted projective space in \cite{Ka} and computed the cohomology ring of weighted projective spaces with integer coefficients. In \cite{HHRW}, Harada-Holm-Ray- Williams introduced divisive weighted projective spaces and computed the integral generalized equivariant cohomology ring of divisive weighted projective spaces. 
Amrani computed the $K$-theory of weighted projective space in \cite{Am2, Am3}. The weighted Grassmannians was introduced by Corti-Reid in \cite{CoRe}, as the weighted projective analog of the Grassmann manifolds. 
The equivariant cohomology rings were studied for weighted Grassmannians in \cite{AbMa1,AbMa2}, and for weighted flag varieties in \cite{ANQ} all with rational coefficients.
%The equivariant cohomology ring of the weighted Grassmannians with rational coefficients was studied in \cite{AbMa1,AbMa2}. The equivariant cohomology ring of weighted Flag varieties with rational coefficients was studied in \cite{ANQ}.
The author and Sarkar \cite{BS} provided another topological definition of weighted Grassmannians and referred to them as weighted Grassmann orbifolds. A description of generalized equivariant cohomologies of divisive weighted Grassmann orbifolds with integer coefficients was studied in \cite{BS}.  For basic properties of orbifolds, readers are referred to \cite{ALR}.

In this article, we explore the Schubert calculus in equivariant $K$-theory of divisive weighted Grassmann orbifolds with integer coefficients by providing a combinatorial description of the Schubert classes in the aforementioned ring. Initially, we derive the Chevalley rule in the equivariant $K$-theory ring of  divisive weighted Grassmann orbifolds. Building on this, we explicitly compute all the structure constants with respect to the Schubert classes and discuss the positivity of the structure constants. Our approach to this computation is to realize Schubert classes as concrete symmetric polynomials. We refer to these polynomials as `twisted factorial Grothendieck polynomials', and give an explicit formula for the restriction of the Schubert classes to any torus fixed point in terms of the twisted factorial Grothendieck polynomials. As a consequence, we introduce twisted Grothendieck polynomials and prove that these polynomials represent the Schubert structure sheaves in the ordinary $K$-theory of the divisive weighted Grassmann orbifold. This paves the way to compute the structure constants in the ordinary $K$-theory of divisive weighted Grassmann orbifolds with integer coefficients. 

% A. Knutson-T. Tao [KnTa03] introduced puzzles to give the first rule for equivariant Schubert calculus that is positive in the sense of [Gr01]. 

The paper is organized as follows.
In Section \ref{bld_seq_wgt_gmn_obd}, we recall the definition of the weighted Grassmann orbifold in terms of the {\Pl} weight vector and discuss the orbifold and $q$-CW complex structure. In Section \ref{Sec_eq_k_th_gsm}, we explore the Schubert calculus in $K_{T^n}(\G(d,n))$, and describe the geometric and combinatorial properties of the Schubert classes in the aforementioned ring, see Proposition \ref{prop_mu_comp}, Proposition \ref{prop_van_pro_S} and Proposition \ref{prop_HHH}. We recall the factorial Grothendieck polynomials from \cite{ikNa} and study their combinatorial properties, see Proposition \ref{prop_div_dif_ope}, Proposition \ref{cor_Gro_pol} and Proposition \ref{prop_van_pro}. We construct an algebraic localization map from an algebra generated by the factorial Grothendieck polynomials to $K_{T^n}(\G(d,n))$ establishing an explicit correspondence between the factorial Grothendieck polynomials and Schubert classes in $K_{T^n}(\G(d,n))$, see Theorem \ref{thm_sur_hom}. 

In Section \ref{sec_sc_cal_pldn}, we explore the equivariant $K$-theory $K_{T^{n+1}}(Pl(d,n))$ of {\Pl} coordinates. We prove that $K_{T^{n+1}}(Pl(d,n))$ is isomorphic to $K_{T^n}(\G(d,n))$ as a $R(T^n)$-algebra and provide a combinatorial description of the Schubert classes in $K_{T^{n+1}}(Pl(d,n))$, see Lemma \ref{lem_div class} and Lemma \ref{lem_pslam}. In Section \ref{Sec_Sc_bas_eq_K_th_wgt_gsm}, we explore the equivariant $K$-theory of divisive weighted Grassmann orbifolds with integer coefficients and realize this as a sub algebra of $K_{T^{n+1}}(Pl(d,n))$. Moreover, we explicitly describe the Schubert classes in the equivariant $K$-theory $K_{T_c}(\G_{\bf c}(d,n))$ of divisive weighted Grassmann orbifolds, see Lemma \ref{lem_sc_bas_wt_case} and Lemma \ref{lem_sc_bas_wt_case_gen}. We also discuss the Schubert classes in the $K$-theory of divisive weighted Grassmann orbifolds.

In Section \ref{sec_twis_gro_pol}, we introduce twisted factorial Grothendieck polynomials. Let $\mathcal{P}_d$ denote the set of partitions of length less than or equal to $d$; that is, $\mathcal{P}_d$ consists of sequences 
$\lambda = (\lambda_1,\lambda_2, \ldots, \lambda_d)$ of non-negative integers satisfying 
$\lambda_1 \ge \lambda_2 \ge \cdots \ge \lambda_d \ge 0.$
 Let $\mathbb{a}=(\mathbb{a}_1,\mathbb{a}_2,\dots)$ be an infinite set of parameter. For every $\lambda \in \mathcal{P}_d$ we introduce `twisted factorial Grothendieck polynomials' $G_{\lambda}^{\bf c}(x|\mathbb{a})$ and
%$$G_\lambda^{\bf c}(x_1,\dots,x_d|\mathbb{a}_1,\mathbb{a}_2,\dots)=G_\lambda(x_1,\dots,x_d|1-\mathbb{a}_1{(\xi(x))^{w_1}},1-\mathbb{a}_1{(\xi(x))^{w_2}},\dots),$$ where $G_\lambda(x_1,\dots,x_d|a_1,a_2,\dots)$ is a factorial Grothendieck polynomial discussed in Section \ref{Sec_eq_k_th_gsm}, and $\xi(x)=\prod_{i=1}^d(1-x_i)$.
describe that $\{G_{\lambda}^{\bf c}(x|\mathbb{a})\}_{\lambda \in \mathcal{P}_d}$ form a basis of the algebra $\ZZ[\mathbb{a}^{\pm1}_1,\mathbb{a}^{\pm1}_2,\dots][x_1,\dots,x_d]^{S_d}$ as a $\ZZ[\mathbb{a}^{\pm1}_1,\mathbb{a}^{\pm1}_2,\dots]$ module. For $0<d<n$, let $\mathcal{P}(d, n)$ denotes the partitions in $\mathcal{P}_d$ contained in a $d\times (n-d)$ rectangle. In other words,
$$\mathcal{P}(d,n)=\{\lambda = (\lambda_1, \ldots, \lambda_d):n-d\ge \lambda_1\ge \lambda_2\ge \dots,\ge \lambda_d\ge 0\}.$$
For every $\lambda \in \mathcal{P}(d,n)$, we describe the Schubert classes ${\bf c}S_\lambda\in K_{T_c}(\G_{\bf c}(d,n))$ and show that $\{{\bf c}S_\lambda\}_{\lambda \in \mathcal{P}(d,n)}$ form a basis of the algebra $K_{T_{\bf c}}(\G_{\bf c}(d,n))$ as $\ZZ[\mathbf{a}^{\pm1}_1,\dots,\mathbf{a}^{\pm1}_n]$ module. 
%Next, we prove the following algebra homomorphism.

\begin{customthm}{A}[Theorem \ref{thm_alg_mor}]
    There exists a surjective algebra homomorphism $$\Psi^{\bf c}\colon \ZZ[\mathbb{a}^{\pm1}_1,\mathbb{a}^{\pm1}_2,\dots][x_1,\dots,x_d]^{S_d}\to K_{T_{\bf c}}(\G_{\bf c}(d,n))$$ of $\ZZ[\mathbb{a}_1^{\pm1},\mathbb{a}_2^{\pm1},\dots]$ to $\ZZ[{\mathbf{a}}_1^{\pm1}, \dots, {\bf a}_n^{\pm1}]$ algebra such that \[\Psi^{\bf c}(G_{\lambda}^{\bf c}(x|\mathbb{a}))=\begin{cases}
{\bf c}S_{\lambda} &\text{ if } \lambda\in \mathcal{P}(d,n)\\
 0 &\text{ otherwise } \\  
\end{cases}.
\]  
\end{customthm}

\begin{customcorollary}{B}[Corollary \ref{cor_alg_mor}]
For every $\mu\in \mathcal{P}(d,n)$, there exist algebra homomorphisms 
$$\Psi^{\bf c}_{\mu}: \ZZ[\mathbb{a}^{\pm1}_1,\mathbb{a}^{\pm1}_2,\dots][x_1,\dots,x_d]^{S_d}\to \ZZ[{\bf{a}}_1^{\pm1}, \dots, {\bf a}_n^{\pm1}]$$ such that for every $\lambda\in \mathcal{P}(d,n)$,
$${\bf c}S_{\lambda}|_{\mu}=\Psi_{\mu}^{\bf c}(G_{\lambda}^{\bf c}(x|\mathbb{a})).$$
\end{customcorollary}
In other words, the restriction of the Schubert class ${\bf c}S_\lambda$ in any torus fixed point can be explicitly computed as a image of twisted factorial Grothendieck polynomial $G_{\lambda}^{\bf c}(x|\mathbb{a})$. Moreover, we define twisted Grothendieck polynomials $G_\lambda^{\bf c}(x)\in \ZZ[x_1,\dots,x_d]^{S_d}$ and the Schubert classes ${\bf c}\mathbb{S}_\lambda\in K(\G_{\bf c}(d,n))$. Interestingly, the collection $\{G_{\lambda}^{\bf c}(x)\}_{\lambda \in \mathcal{P}_d}$ form a basis of the $\ZZ$-algebra $\ZZ[x_1,\dots,x_d]^{S_d}$, and $\{{\bf c}\mathbb{S}_\lambda\}_{\lambda \in \mathcal{P}(d,n)}$ forms a basis of the $\ZZ$-algebra $K(\G_{\bf c}(d,n))$. Then we prove the following algebra homomorphism.

\begin{customthm}{C}[Theorem \ref{thm_alg_hom_ord_case}]
There exists a surjective homomorphism of the $\ZZ$-algebra
$$\ZZ[x_1,\dots,x_d]^{S_d}\to K(\G_{\bf c}(d,n))$$ 
  which sends $G_{\lambda}^{\bf c}(x)$ to ${\bf c}\mathbb{S}_\lambda$ if $\lambda\in \mathcal{P}(d,n)$, and to 0 if $\lambda\in \mathcal{P}_d\setminus \mathcal{P}(d,n)$.  
\end{customthm}

Moreover, the twisted Grothendieck polynomials can be written as the $\ZZ$-linear combination of Grothendieck polynomials, see Theorem \ref{thm_Z_lin_com}.

 In Section \ref{sec_che_rule}, we prove the Chevalley rule, the multiplication rule ${\bf c}S_{\lambda}{\bf c}S_{(1)}$, where $(1)$ denote the unique partition $(1,0,\dots,0)\in \mathcal{P}(d,n)$ of total size 1.
\begin{customthm}{D}[Theorem \ref{thm_eq_che}][Chevalley rule]
\begin{equation*}
{\bf c}S_{\lambda}{\bf c}S_{(1)}=(1-\frac{\mathbf{a}_{(0)}}{(\mathbf{a}_{\lambda})^{d_\lambda}}){\bf c}S_\lambda-{\mathbf{a}_{(0)}}\sum_{\mu:\lambda\xRightarrow[d_\lambda]{}^*\mu}\mathcal{L}_{\lambda, d_{\lambda}}^\mu{\bf c}S_{\mu}.
\end{equation*}
\end{customthm}
The formula to compute the Laurent polynomial $\mathcal{L}_{\lambda,d_{\lambda}}^\mu$ is explained in \eqref{eq_lau_pol}.

In Section \ref{sec_int_coh}, we concentrate to compute the structure coefficient ${\bf c}{K}_{\lambda\mu}^{\nu}$ defined by:
$$
{\bf c}{S}_{\lambda}{\bf c}{{S}}_{\mu}=\sum_{\eta\in \mathcal{P}(d,n) }{\bf c}{K}_{\lambda\mu}^{\eta}{\bf c}{{S}}_{\eta}.$$

\begin{customthm}{E}[Theorem \ref{str_cst_eq_case}]
$${\bf c}{K}_{\lambda\mu}^\eta=\sum_{\nu:\nu \succeq\lambda,\mu  }\sum_{I: \nu \xRightarrow[d_{I,\nu}]{} \eta}C(\lambda,\mu,\nu,I)\mathcal{U}_I \mathcal{L}_{\nu,d_{I,\nu}}^\eta.$$
\end{customthm}
Here $I$ is a finite collection of elements of $\{1,\dots,n-1\}$, $\mathcal{U}_I=\prod_{i\in I}\frac{{\bf a}_{i}}{{\bf a}_{i+1}}$ and the coefficients $C(\lambda,\mu,\nu,I)\in \ZZ$ are known due to \cite{PY}. We describe the formula to compute the Laurent polynomial $\mathcal{L}_{\nu,d_{I,\nu}}^\eta$ in \eqref{eq_lau_pol}. 
As a corollary, we compute the structure constants $ {\bf{c}}{\mathscr{K}}_{\lambda\mu}^{\eta}$ of ordinary $K$-theory of divisive weighted Grassmann orbifolds with respect to Schubert classes. 

\begin{customcorollary}{F}[Corollary \ref{cor_st_con_ord_case}]
   $$ {\bf{c}}{\mathscr{K}}_{\lambda\mu}^{\eta}={\mathscr{K}}_{\lambda\mu}^{\eta}+\sum_{\nu:~\eta\succ\nu \succeq\lambda,\mu  }\sum_{I: \nu \xRightarrow[d_{I,\nu}]{} \eta}(-1)^{|\eta\setminus \nu|}C(\lambda,\mu,\nu,I)N_{\nu,d_{I,\nu}}^{\eta}.$$ 
\end{customcorollary}

The coefficients $N_{\nu,d_{I,\nu}}^{\eta}$ is explained in Remark \ref{rem_a_0}. In addition, we explore the positivity of the structure constants ${\bf c}{K}_{\lambda\mu}^{\eta}$, see Theorem \ref{thm_postivity}. Some explicit computations of the structure constants are discussed in Example \ref{eg_wgt_pro_sp} and Example \ref{eg_wgt_gr_24}.

\section{An overview of weighted Grassmann orbifolds} \label{bld_seq_wgt_gmn_obd}

In this section, we explore Plücker weight vector and revisit the definition of the weighted Grassmann orbifold, formulated as a orbit space of Plücker coordinates under a $\CC^*$-action determined by the Plücker weight vector following \cite{br}. This construction aligns with the notion of the weighted Grassmannian as presented in \cite{CoRe,AbMa1,BS}. Furthermore, we discuss the CW complex structure of divisive weighted Grassmann orbifolds.

%Note that there are $n\choose d$ many elements in the set $\mathcal{P}(d, n)$.   

For two positive integer $d<n$, let $I(d,n)$ denote the set of all cardinality $d$ subset $(x_1,x_2,\dots,x_d)$ in $\{1,2,\dots,n\}$ such that $1\leq x_1 < x_2 < \cdots <x_d\leq n$. The elements of $I(d,n)$ are known as the Schubert symbols; see \cite{BS,AbMa1}. There is a standard bijection between the partitions in $\mathcal{P}(d, n)$ the Schubert symbols ${I(d,n)}$ given by $$(\lambda_1,\lambda_2,\dots,\lambda_d) \mapsto (\lambda_d+1,\lambda_{d-1}+2,\dots,\lambda_1+d).$$

  Let $\lambda=(\lambda_1,\lambda_2,\dots,\lambda_d)$ be a partition in $\mathcal{P}(d, n)$. For each $i\in \{1,\dots,d\}$, define 
  \begin{equation}\label{eq_cor}
     \widetilde{\lambda}_{i}:=\lambda_{d+1-i}+i.  
  \end{equation}
 Then $\widetilde{\lambda}_i\in \{1,2,\dots,n\}$ and $1\leq \widetilde{\lambda}_1 < \widetilde{\lambda}_2 < \cdots < \widetilde{\lambda}_d\leq n$. 
 Thus  $(\widetilde{\lambda}_1,\widetilde{\lambda}_2,\dots,\widetilde{\lambda}_d)$ is the Schubert symbol corresponding to $\lambda$.
 \begin{remark}
Throughout the paper, we use the same notation $\lambda$ to denote the elements in $\mathcal{P}(d,n)$ as well as in $I(d,n)$. If we consider $\lambda\in \mathcal{P}(d,n)$ we denote this by $\lambda=({\lambda}_1,{\lambda}_2,\dots,{\lambda}_d)$. If we consider $\lambda\in I(d,n)$ we denote $\lambda=(\widetilde{\lambda}_1,\widetilde{\lambda}_2,\dots,\widetilde{\lambda}_d)$. The relation between $\lambda_i$ and $\widetilde{\lambda}_i$ is given in \eqref{eq_cor}.
 \end{remark}

Define a partial order `$\preceq$' on the set $I(d,n)$ by
\begin{equation}\label{eq_bru_ord}
	\lambda \preceq \mu \text{ if } \widetilde{\lambda}_i\leq \widetilde{\mu}_i \text{ for all } i=1,2,\dots,d.
\end{equation}

In addition, the dictionary order `$\leq$' gives a total order on the set $I(d,n)$. This induces a total order on $I(d,n)$ and it satisfies if $\lambda\preceq\mu$ then $\lambda\leq \mu$.

% Assume $\lambda\in I(d,n)$ be a Schubert symbol and $(i,j)$ be a pair such that $i\in \lambda$ and $j\notin\lambda$ then we define $(i,j)\lambda$ be the Schubert symbol obtained by replacing $i$ by $j$ in $\lambda$ and ordering the later set.
%  A pair $(s,s')$ is called a reversal of $\lambda\in I(d,n)$ if $s\in \lambda, s'\notin \lambda$ and $s>s'$. The set of all reversals of $\lambda$ is denoted by $\rm{rev}(\lambda)$.  Now corresponding to $\text{rev}(\lambda)$, one can define a subset of Schubert symbols as follows 
% \begin{equation}\label{eq_inv_lbd}
% 	R(\lambda):=\{\mu~|~\mu=(s,s')\lambda \text{ for } (s,s')\in \text{rev}(\lambda)\}.
% \end{equation}

\subsection{Pl\"{u}cker weight vectors and weighted Grassmann orbifolds}\label{subsec_pl_wgt_vec}
Let $\Lambda^d(\CC^n)$ be the $d$-th exterior product of the complex $n$-dim space $\CC^n$. The standard basis $\{f_1,\dots, f_n\}$ for $\mathbb{C}^n$ induces a basis $\{f_{\lambda}\}_{\lambda \in I(d,n)}$ of $\Lambda^d(\mathbb{C}^n)$, where $f_{\lambda}=f_{\tilde{\lambda}_1}\wedge f_{\tilde{\lambda}_2}\wedge\dots\wedge f_{\tilde{\lambda}_d}$ and $\lambda=(\tilde{\lambda}_1,\tilde{\lambda}_2,\dots,\tilde{\lambda}_d) \in I(d,n)$. Consider a subset $Pl(d,n)\subseteq\Lambda^d(\CC^n)$ by the following
\begin{equation}\label{eq_pl_kn}
   Pl(d,n):= \{a_1\wedge a_2\wedge\dots\wedge a_d\colon a_i\in \CC^n\}\setminus\{\bf{0}\}. 
\end{equation}

The elements of $Pl(d,n)$ are known as {\Pl} coordinates, those are also defined as
solutions to a system of homogeneous polynomial equations known as {\Pl} relations, see \cite[Theorem 3.4.11]{Jac} and \cite[Section 2.1]{br} for more details.

 \begin{definition}\label{def_plu_vec}
 A weight vector ${\bf{c}}=(c_\lambda)_{\lambda \in I(d,n)}\in (\mathbb{Z}_{\geq 1})^{n \choose d}$ is called a Pl\"{u}cker weight vector if for any two ordered sequences $1\leq i_1<\dots<i_{d-1}\leq n$ and $1\leq \ell_0<\ell_1<\dots<\ell_d\leq n$ the following satisfied
$$c_{\lambda^j}+c_{\mu^j}= c_{\lambda^i}+c_{\mu^i} \text{ for all } i,j\in \{0,1,\dots,d\},$$ where 
$\lambda^{j}=(i_1,\dots,i_{d-1},\ell_j) \in I(d,n)$ and $\mu^{j}=(\ell_0,\dots,\widehat{\ell}_j,\dots,\ell_d) \in I(d,n)$.
 \end{definition}

% \begin{remark}\label{rem_plu_wgt_vec}
% Let ${\bf c}=(c_\lambda)_{\lambda\in \mathcal{P}(d,n)}$ be a Pl\"{u}cker weight vector. If  $\lambda\cup \mu=\bar{\lambda}\cup \bar{\mu}$ for any 4 Schubert symbols $\lambda$, $\mu$, $\bar{\lambda}$, $\bar{\mu}$ then $c_\lambda+c_\mu=c_{\bar{\lambda}}+c_{\bar{\mu}}$.  
% \end{remark}

\begin{example}\label{ex_pl_wgt_vec}
	If $d=2$ and $n=4$, then $\mathcal{P}(2,4)$ consists of 6 elements $(0,0),(1,0),$ $(2,0),(1,1),(2,1),(2,2)$. The corresponding elements in $I(2,4)$ are given by $(1,2),(1,3),$ $(1,4),(2,3),(2,4),(3,4)$ respectively. Consider order sequences  $i_1=1,~ \ell_0=2<\ell_1=3<\ell_2=4.$
    A weight vector ${\bf{c}}=(c_{(1,2)},c_{(1,3)},c_{(1,4)},c_{(2,3)},c_{(2,4)},c_{(3,4)})\in (\ZZ_{\geq 1})^6$ is a Pl\"{u}cker weight vector if $$c_{(1,2)}+c_{(3,4)}=c_{(1,3)}+c_{(2,4)}=c_{(1,4)}+c_{(2,3)}.$$  
\end{example}

%$Pl(d,n)$ is $\CC^*$-invariant subset of $\CC^{m+1}\setminus\{{\bf 0}\}$ under the ${\bf c}$-action iff ${\bf c}$ is a Pl\"{u}cker weight vector. Then we can define the following.

\begin{definition}\label{def_wgt_gsm}
Let ${\bf{c}}=(c_\lambda)_{\lambda \in I(d,n)}\in (\mathbb{Z}_{\geq 1})^{n \choose d}$ be a Pl\"{u}cker weight vector. Define a `${\bf c}$-action' of $\CC^*$ on $Pl(d,n)$ by $$t{\cdot}(\zeta_\lambda)_{\lambda\in I(d,n)}=(t^{c_\lambda}\zeta_\lambda)_{\lambda\in I(d,n)}.$$
 We denote the orbit space  $$\G_{\bf{c}}(d,n):=\frac{Pl(d,n)}{{\bf c}\text{-action}}.$$
	Consider the quotient map  $$\pi_{\bf{c}}: Pl(d,n)\to \G_{\bf{c}}(d,n).$$
    The topology on $\G_{\bf{c}}(d,n)$ is given by the quotient topology induced by the map $\pi_{\bf{c}}$. 
\end{definition}

Note that $Pl(d,n)$ is stable with respect to ${\bf c}$-action of $\CC^*$ if and only if ${\bf c}$ is a {\Pl} weight vector, see \cite{br} for more details.

\begin{remark}
  If ${\bf{c}}=(1,1,\dots,1)$, then ${\bf{c}}$ is a  Pl\"{u}cker weight vector. In this case the ${\bf c}$-action of $\CC^*$ reduced to
  $$t{\cdot}(\zeta_\lambda)=(t\zeta_\lambda)$$
  and the space $\G_{\bf{c}}(d,n)$ becomes the Grassmann manifold $\G(d,n)$. We denote the corresponding quotient map by $$\pi: Pl(d,n)\to \G(d,n).$$
\end{remark}
  
%The definition of weighted Grassmannians was introduced in \cite{CoRe}.
{\Pl} weight vectors appears naturally as follows. Let $W:=(w_1, \dots ,w_n) \in (\mathbb{Z}_{\geq 0})^{n}$ and $a \in\mathbb{Z}_{\geq 1}$. For all $\lambda\in I(d,n)$, define 
\begin{equation}\label{eq_wli}
	w_\lambda:=a+\sum_{j=1}^{d} w_{\tilde{\lambda}_j},
\end{equation}
%where $(\tilde{\lambda}_1,\tilde{\lambda}_2,\dots,\tilde{\lambda}_d)$ is the element in $I(d,n)$ corresponding to $\lambda$. 
Then $w_\lambda\geq 1$ for all $\lambda\in I(d,n)$. From definition \ref{def_plu_vec}, it follows that ${\bf{c}}=(c_\lambda)_{\lambda\in I(d,n)}$ is a {\Pl} weight vector, where $c_\lambda=w_\lambda$. Conversely, we have the following.
\begin{proposition}\label{prop_rel_two_wt}\cite[Proposition 2.9]{br}
Let ${\bf{c}}=(c_\lambda)_{\lambda\in I(d,n)}\in (\mathbb{Z}_{\geq 1})^{n\choose d}$ be a Pl\"{u}cker weight vector for $d<n$. Then there exist $a\in \{1,2,\dots,d\}$ and $W=(w_1,\dots,w_n) \in \mathbb{Z}^{n}$ such that $c_\lambda=w_{\lambda}$ for every element $\lambda\in I(d,n)$, where $w_{\lambda}$ is defined in \eqref{eq_wli}.
\end{proposition}

% \begin{remark}\label{lem_many_WD}
% Let $W=(w_1,\dots,w_n)\in \ZZ^n$ and $a\in \ZZ_{\ge 1}$. The Pl\"{u}cker weight vector corresponding to $(W,a)$ is same as the Pl\"{u}cker weight vector corresponding to $(W',a')$, where $W'=(w_1-\alpha,\dots,w_n-\alpha)\in \ZZ^n$ and $a'=a+d\alpha$ for any $\alpha \in \ZZ_{\geq 0}$.
% \end{remark}

 The algebraic torus $(\CC^*)^{n+1}$ acts on $Pl(d,n)$ by the following: 
\begin{equation}\label{eq_c_action}
	(t_1, t_2, \dots, t_n, t) \sum_{\lambda\in I(d,n)}\zeta_{\lambda}f_{\lambda}=\sum_{\lambda\in I(d,n)} t{\cdot}t_{\lambda} \zeta_{\lambda} f_{\lambda}, 
\end{equation}	
where $t_\lambda:=t_{\tilde{\lambda}_1}t_{\tilde{\lambda}_2}\dots t_{\tilde{\lambda}_d}$ for $\lambda=(\tilde{\lambda}_1,\dots,\tilde{\lambda}_d)\in I(d,n)$. Let $W=(w_1, \dots ,w_n) \in (\mathbb{Z}_{\geq 0})^{n}$, $a \in\mathbb{Z}_{\geq 1}$ and consider a subgroup $WD$ of $(\CC^*)^{n+1}$ defined by \begin{equation}\label{eq_wd}
	   	WD:=\{(t^{w_1},\dots,t^{w_n},t^a):t\in \CC^*\}.  
\end{equation}
The restricted action of $WD$  on $Pl(d,n)$ is the same as the  ${\bf{c}}$-action of $\CC^*$ on $Pl(d,n)$, where the Pl\"{u}cker weight vector ${\bf{c}}=(c_\lambda)_{\lambda\in I(d,n)}$ is defined using $c_\lambda=w_\lambda$ as in  \eqref{eq_wli}. Thus $\G_{\bf{c}}(d,n)$ is same as the orbit space $\frac{Pl(d,n)}{WD}$, which is referred to as a weighted Grassmannian $\WG(d,n)$ in \cite{AbMa1}.

    In \cite{BS}, the author and Sarkar introduced a topological definition of $\WG(d,n)$ and called it weighted Grassmann orbifold. Using the argument of \cite[Subsection 2.2]{AbMa1} and Proposition \ref{prop_rel_two_wt}, the quotient space $\G_{\bf{c}}(d,n)$ has an orbifold structure for a Pl\"{u}cker weight vector ${\bf{c}}$. We call the space $\G_{\bf{c}}(d,n)$ a weighted Grassmann orbifold associated to the Pl\"{u}cker weight vector ${\bf{c}}$.

\subsection{CW complex structures of divisive weighted Grassmann orbifolds}\label{subsec_q_cell_stu}
  
In this subsection, we recall the $q$-CW complex structure of $\G_{\bf{c}}(d, n)$ following \cite{BS,AbMa1}. We define a divisive weighted Grassmann orbifold and discuss that it has a CW complex structure.

A CW complex structure of $\G(d,n)$ given by 
\begin{equation*}
  \G(d,n)=\bigsqcup_{\lambda\in I(d,n)} {E}(\lambda). 
\end{equation*}
The Schubert cell $E(\lambda)$ is defined by:
\begin{equation}\label{eq_open_cell}
E(\lambda)=\{[u_1\wedge\dots\wedge u_d]\in \G(d,n) ~|~  u_{i\widetilde{\lambda}_i}=1 ;~u_{ij}=0 \text{ if }j<\widetilde{\lambda}_i \text{ or, } j=\widetilde{\lambda}_{i+1},\dots,\widetilde{\lambda}_d\},
\end{equation}
where $u_i=(u_{i1},\dots,u_{in})\in \CC^n$ for $1\leq i\leq d$. Then $E(\lambda)$ is an open cell of codimension $|\lambda|=\lambda_1+\dots+\lambda_d$.  Define \begin{equation*}Y_\lambda:=\sqcup_{\mu\succeq \lambda} E(\mu). 
\end{equation*}
Then $Y_\lambda=\overline{E(\lambda)}$ (the Zariski closure), is known as the Schubert variety. The structure sheaf in $K_{T^n}(\G(d,n))$ corresponding to the Schubert variety $Y_\lambda$ is known as the Schubert class. We discuss more on the combinatorial description of Schubert class in Section \ref{Sec_eq_k_th_gsm}.

Let $G(c_\lambda)$ be a finite subgroup of $\CC^*$ defined by 
$$G(c_\lambda)=\{t\in \CC^*:t^{c_\lambda}=1\}.$$

A $q$-CW complex structure of $\G_{\bf c}(d,n)$ is given by 

\begin{equation*}
\G_{\bf{c}}(d,n)=\bigsqcup_{\lambda\in I(d,n)}\frac{E(\lambda)}{G(c_\lambda)},
\end{equation*}
see \cite{BS,br}, where $G(c_\lambda)$ action on $E(\lambda)$ is given by 
\begin{align}\label{eq_t_act}
    t[(z_\mu)_{\mu\in I(d,n)}]=[(t^{c_\mu}z_\mu)_{\mu\in I(d,n)}] \text{ for } t\in G(c_\lambda).
\end{align}

%as a restriction of the ${\bf c}$-action of $\CC^*$.

For every permutation $\sigma$ on the set $I(d,n)$ and $z=(z_\lambda)_{\lambda\in I(d,n)}\in Pl(d,n)$ define $\sigma z:=(z_{\sigma(\lambda)})_{\lambda\in I(d,n)}$. We  introduce a sign vector $t:=(t_{\lambda})_{\lambda\in I(d,n)}$ such that $t_\lambda\in \{1,-1\}$ for all $\lambda\in I(d,n)$.
\begin{definition}
A permutation $\sigma$ on the set $I(d,n)$ is said to be a Pl\"{u}cker permutation if there exist a sign vector $t_\sigma$ such that $t_\sigma{\cdot}\sigma z\in Pl(d,n)$ for every $z\in Pl(d,n)$. 
\end{definition}
Similarly, for every {\Pl} weight vector ${\bf c}=(c_\lambda)_{\lambda\in I(d,n)}$, define $\sigma {\bf c}:=(c_{\sigma(\lambda)})_{\lambda\in I(d,n)}$.
\begin{proposition}\label{prop_pl_per}
   Let ${\bf c}$ be a Pl\"{u}cker weight vector and $\sigma$ be a {\Pl} permutation. Then $\sigma{\bf c}$ is a {\Pl} weight vector. Moreover, $\G_{\bf c}(d,n)$ is homeomorphic to $\G_{\sigma \bf c}(d,n)$.
\end{proposition}

\begin{remark}
Using the homeomorphism in Proposition \ref{prop_pl_per}, the weighted Grassmann orbifold $\G_{\bf c}(d,n)$ has another $q$-CW complex structure given by $$\G_{\bf c}(d,n)=\bigsqcup_{\lambda\in I(d,n)}\frac{E(\lambda)}{G(c_{\sigma(\lambda)})}.$$    
\end{remark}

\begin{definition}\cite{BS}
A weighted Grassmann orbifold $\G_{\bf c}(d,n)$ is said to be divisive if there exists a {\Pl} permutation $\sigma$ such that $c_{\sigma(\lambda)}$ divides $c_{\sigma(\mu)}$ for all $\mu \leq \lambda$, where `$\leq$' is the total order on $I(d,n)$ induced from the dictionary order. 
\end{definition}

From \eqref{eq_open_cell} it follows that, if $[z]=[(z_\mu)_{\mu\in I(d,n)}]\in E(\lambda)$ then $z_\mu=0$, whenever $\mu\nsucceq\lambda$. The $G(c_\lambda)$ action on $E(\lambda)$ as described in \eqref{eq_t_act} becomes trivial if $c_\lambda$ divides $c_\mu$ for $\mu\succeq \lambda$. Thus, if $\G_{\bf c}(d,n)$ is a divisive weighted Grassmann orbifold, then one can find a {\Pl} permutation $\sigma$ such that the action of $G(c_{\sigma(\lambda)})$ on $E(\lambda)$ becomes trivial action. Moreover, we have the following:

\begin{proposition}\cite[Theorem 3.19]{BS}\label{thm_CW_cplx_str}
If $\G_{\bf c}(d,n)$ is a divisive weighted Grassmann orbifold, then it has a CW complex structure with even-dimensional cells $\{E(\lambda)\}_{\lambda\in I(d,n)}$. %Thus, the integral cohomology of the divisive weighted Grassmann orbifold is torsion-free and concentrated in even degrees. 
\end{proposition}

If $\G_{\bf c}(d,n)$ is a divisive weighted Grassmann orbifold then using Proposition \ref{prop_pl_per}, it is enough to assume that $c_{\lambda}$ divides $c_{\mu}$ for $\mu \leq \lambda$. Throughout the paper, we assume this condition whenever we say  $\G_{\bf c}(d,n)$ is a divisive weighted Grassmann orbifold. Moreover, we  define $d_{\mu \lambda}:=\frac{c_\mu}{c_\lambda}$ if $\mu\leq \lambda$ and $d_{\lambda}:=\frac{c_{(0)}}{c_\lambda}$, where $(0)$ is the unique element $(0,0,\dots,0)\in \mathcal{P}(d,n)$.
We study Schubert calculus on equivariant $K$-theory of divisive weighted Grassmann orbifold in Section \ref{Sec_Sc_bas_eq_K_th_wgt_gsm} onwards.

\section{Factorial Grothendieck polynomials and equivariant $K$-theory of Grassmannians}\label{Sec_eq_k_th_gsm}
In this section, we provide combinatorial description of the Schubert classes in the equivariant $K$-theory of Grassmann manifolds. We recall the factorial Grothendieck polynomials, following \cite{ikNa, mc}, and discuss their combinatorial and geometric properties. Subsequently, we prove connection between the Schubert classes and the factorial Grothendieck polynomials through an algebraic localization map.

\subsection{The Schubert classes in the equivariant $K$-theory of Grassmannians}\label{subsec_Sch_class}
In this subsection, we provide a combinatorial description of the Schubert classes in the GKM ring $K_{T^n}(\G(d,n))$ as a $R(T^n)$-algebra. We explicitly describe each Schubert classes in terms of divided difference operator $\pi_i$ acts on $K_{T^n}(\G(d,n))$. Define a $T^n$-action on $\G(d,n)$ by
$$(t_1,\dots,t_n)[(\zeta_\lambda)_{\lambda\in I(d,n)}]=[(t_{\lambda}\zeta_\lambda)_{\lambda\in I(d,n)}],$$
where $t_\lambda=t_{\tilde{\lambda}_1}t_{\tilde{\lambda}_2}\dots t_{\tilde{\lambda}_d}$. The equivariant $K$-theory $K_{T^n}(\G(d,n))$ has a module structure over $K_{T^n}(\{pt\})=R(T^n)$. We denote the representation ring by $R(T^n)=\ZZ[{a}_1^{\pm1}, \dots, {a}_n^{\pm1}]$. For every partition $\lambda\in I(d,n)$, define $a_\lambda=a_{\tilde{\lambda}_1}a_{\tilde{\lambda}_2}\dots a_{\tilde{\lambda}_d}$. 
\begin{proposition}
    The $T^n$-equivariant $K$-theory of $\G(d,n)$ is given by 
   {\small $$K_{T^n}(\G(d,n))=\Big\{f\in \bigoplus_{\lambda\in I(d,n)} \ZZ[{a}_1^{\pm1}, \dots, {a}_n^{\pm1}]~~\Big{|}~~ f|_\mu-f|_\lambda \text{ is divisible by } (1-\frac{{a}_{\lambda}}{{a}_{\mu}})\text{ if } \mu=(i,j)\lambda \Big\}.$$}
\end{proposition}

The symmetric group $S_n$ acts on $ \ZZ[a_1^{\pm1},\dots,a_n^{\pm1}]$ by permuting the variables $a_i$, also it acts on the set $ I(d,n)$ of all Schubert symbol, as every element of $I(d,n)$ is a subset of $\{1,2,\dots,n\}$. 
We focus on the simple reflections $s_i:= (i,i+1)\in S_n$ for $1\leq i\leq n-1$. For every $\lambda \in I(d,n)$, we have 
    $s_i\lambda\prec\lambda$ (resp. $s_i\lambda\succ\lambda$) if $i+1\in \lambda$, $i\notin \lambda$ (resp. $i\in \lambda, i+1\notin \lambda$).
If both  $i,i+1$ is in $\lambda$ or, none of  $i,i+1$ is in $\lambda$ then $s_i\lambda=\lambda$ .

We define the divided difference operators $\pi_i$ on $K_{T^n}(\G(d,n))$. Let $x\in K_{T^n}(\G(d,n))$, define $\pi_ix$ by
 %$$\pi_ix|_\mu:= x(\mu)-\frac{(x(\mu)- s_i( x(s_i\mu))}{1-\frac{a_{i+1}}{a_i}},$$
% where $e_i(\mu)=1-\frac{a_{i+1}}{a_i}$. 
$$\pi_ix|_\mu:= \frac{x|_\mu-(1-e_i(a))s_i( x|_{s_i\mu})}{e_i(a)},$$
where $e_i(a):=1-\frac{a_{i}}{a_{i+1}}$.
Apriori, this is just a rational function. But $\{\pi_ix\}_{i=1}^{n-1}$ are actually elements of $K_{T^n}(\G(d,n))$ if $x\in K_{T^n}(\G(d,n))$. This follows using a similar argument as in \cite[Appendix, 1st Lemma]{KnTa}. 
%The operator $\pi_i$ satisfies the following relation $$\pi_i\pi_{i+1}\pi_i=\pi_{i+1}\pi_i\pi_{i+1}~ \text{if }(1\leq i\leq n-2); \quad\pi_i\pi_j=\pi_j\pi_i \text{ if }|i-j|\geq 2.$$

% \begin{remark}
%     The operator $\pi_i$ can be written as $\pi_i=\partial_i+1$, where $$\partial_i=\frac{s_i-1}{e_i(a)}.$$
%     Here $s_i$ acts on $K_{T^n}(\G(d,n))$ by $s_ix|_\mu=s_i( x|_{s_i\mu})$, for $x\in K_{T^n}(\G(d,n))$. The operator $\partial_i$ was defined in \cite{dem} and known as Demazure operator in the literature.
% \end{remark}

\begin{definition}\label{def:sch bas}
    A set of elements $\{S_{\lambda}: \lambda \in \mathcal{P}(d,n)\}\subset K_{T^n}(\G(d,n))$ is said to be the family of Schubert classes if the following conditions are satisfied.
\begin{enumerate}
    \item $\pi_i S_{\lambda}=
  \begin{cases}
      S_{s_i\lambda} & \text{ if } s_i\lambda\prec\lambda\\
      S_\lambda & \text{ if } s_i\lambda \succeq \lambda
  \end{cases}\\$ 

  \item $  S_\lambda|_{(0)}=\delta_{\lambda,(0)}$ (the Kronecker's delta).
\end{enumerate} 
  
\end{definition}

A family of Schubert class exists. The existence of Schubert class also follows as a consequence of Theorem \ref{thm_sur_hom}. The uniqueness of the Schubert class follows from Definition \ref{def:sch bas}. Moreover, for every $\mu\in \mathcal{P}(d,n)$, the $\mu$-th coordinate of $S_\lambda$ is given by the following recurrence relation.

\begin{proposition}\label{prop_mu_comp}
Let $\{S_{\lambda}:\lambda\in \mathcal{P}(d,n)\}$ be a family of Schubert classes. If $\mu\neq (0)$ be a Schubert symbol and $s_i\mu \prec\mu$ for some $i$ then 
\[S_{\lambda}|_\mu=\begin{cases}
(1-e_i(a))s_i(S_{\lambda}|_{s_i\mu})+e_i(a)s_i(S_{s_i\lambda}|_{s_i\mu}) & \text{ if } s_i\lambda\prec\lambda\\

    s_i(S_\lambda|_{s_i\mu})& \text{ if } s_i\lambda\succeq \lambda   
\end{cases}.\]
\end{proposition}

\begin{remark}
The recurrence relation in Proposition \ref{prop_mu_comp} called `left hand' recurrence in \cite[Remark 2.3]{LSS}. This recurrence is also given in \cite[Lemma 5.1]{ikNa} other classical types.    
\end{remark}

 A pair $(s,s')$ is called a reversal of $\lambda\in I(d,n)$ if $s\in \lambda, s'\notin \lambda$ and $s>s'$. The set of all reversals of $\lambda$ is denoted by $\rm{rev}(\lambda)$.  Now corresponding to $\text{rev}(\lambda)$, one can define a subset of Schubert symbols as follows 
\begin{equation}\label{eq_inv_lbd}
	R(\lambda):=\{\mu~|~\mu=(s,s')\lambda \text{ for } (s,s')\in \text{rev}(\lambda)\}.
\end{equation}
Here $(s,s')\lambda$ be the Schubert symbol obtained by replacing $s$ by $s'$ in $\lambda$ and ordering the later set. Thus $(1-\frac{a_{s'}}{a_s})=(1-\frac{a_\mu}{a_\lambda})$ if $\mu=(s,s')\lambda$.
\begin{proposition}\label{prop_van_pro_S}
    For every $\lambda\in \mathcal{P}(d,n)$ the elements $S_{\lambda} \in K_{T^n}(\G(d,n))$ satisfies the following conditions
    \begin{enumerate}
    \item  ${S}_{\lambda}|_{\mu}= 0 \text{ if }\mu \nsucceq \lambda$.\\
    \item $ {S}_{\lambda}|_\lambda=\prod_{\nu\in R(\lambda)}(1-\frac{{a}_{\nu}}{{a}_{\lambda}}).$
 \end{enumerate}
\end{proposition}

\begin{proposition}\label{prop_HHH}
   $\{S_\lambda:\lambda \in \mathcal{P}(d,n)\}$ forms a $R(T^n)$-module basis of $K_{T^n}(\G(d,n))$. 
\end{proposition}  
\begin{proof}
    This follows from \cite[Proposition 4.1]{HHH}.
\end{proof}

Let $(0)$ denote the element $(0,\dots,0)\in \mathcal{P}(d,n)$, corresponding to $(1,2,\dots,d)\in I(d,n)$. $(1)$ denotes the element $(1,0,\dots,0)\in \mathcal{P}(d,n)$, corresponding to $(1,2,\dots,d-1,d+1)\in I(d,n)$. Thus $$a_{(0)}=a_{1}a_{2}\dots a_d,\quad a_{(1)}=a_{1}a_{2}\dots a_{d-1}a_{d+1}.$$
\begin{lemma}\label{lem_div_case}
   The Schubert divisor class $S_{(1)}$ is given by $S_{(1)}|_\mu=1-\frac{a_{(0)}}{a_{\mu}}$.
\end{lemma}
\begin{proof}
If $\mu=(0)$ then $S_{(1)}|_\mu=1-1=0$, that is consistence with Definition \ref{def:sch bas}. For $\mu=(1)$ it follows from Proposition \ref{prop_van_pro_S}. For the remaining $\mu$ the proof follows using induction on $|\mu|$ and the recurrence relation as in Proposition \ref{prop_mu_comp}. Let $\lambda=(1)$. For any $\mu\succ \lambda$, there exist $i\neq d$ such that $s_i\mu\prec\mu$. By induction hypothesis, $S_{(1)}|_{s_i\mu}=(1-\frac{a_{(0)}}{a_{s_i\mu}})$. If $i\neq d$ then $s_i\lambda \succeq \lambda$ and $s_i(a_{(0)})=a_{(0)}$. Thus, using  Proposition \ref{prop_mu_comp},  
\begin{equation*}
  S_{(1)}|_\mu=s_i(1-\frac{a_{(0)}}{a_{s_i\mu}})=(1-\frac{a_{(0)}}{a_{\mu}}). 
\end{equation*}
\end{proof}

%Note that one can construct the Schubert class $S_\lambda$ by starting with the class $S_{(1)}$ and  and applying successive divided difference operators following Proposition \ref{prop_mu_comp}.

\begin{lemma}\label{lem_loc_sc_bas_gsm}
  $\{S_{\lambda}|_\mu:\lambda \in \mathcal{P}(d,n)\}$ is a polynomial in $\frac{a_{\nu}}{a_\mu}$ for some $\nu\preceq \mu$.
\end{lemma}
\begin{proof}
We proceed by induction on $|\mu|$. For $\mu=(0)$, the statement is true using Proposition \ref{prop_van_pro_S}. For $\mu\succ (0)$ consider $s_i$ such that $s_i\mu \prec\mu$. By induction hypothesis $S_{\lambda}|_{s_i\mu}$ is a polynomial in $\frac{a_{\eta}}{a_{s_i\mu}}$ for some $\eta\preceq s_i\mu$.  Now $\eta\preceq s_i\mu$ and $s_i\mu\prec\mu$ together implies $s_i\eta\preceq \mu$. Thus $s_i(S_{\lambda}|_{s_i\mu})$ is a polynomial in $\frac{a_{\nu}}{a_{\mu}}$ for some $\nu\preceq \mu$, where $\nu=s_i\eta$. Moreover, $s_i(S_{s_i\lambda}|_{s_i\mu})$ is a polynomial in $\frac{a_{\nu}}{a_{\mu}}$ for some $\nu\preceq \mu$ by the similar argument. Also, $s_i\mu\prec\mu$ implies $i+1\in \{\widetilde{\mu}_1,\dots,\widetilde{\mu}_d\}$ but not $i$. Thus $e_i(a)=1-\frac{a_i}{a_{i+1}}$ can be written as $\frac{a_\nu}{a_\mu}$ for some $\nu\preceq \mu$. Hence, using the recurrence relation as in Proposition \ref{prop_mu_comp}, we can write $S_\lambda|_\mu$ as a polynomial in $\frac{a_\nu}{a_\mu}$ for $\nu\preceq \mu$.
%${S}_{\lambda}|_\lambda=\prod_{\nu\in R(\lambda)}(1-\frac{{a}_{\nu}} {{a}_{\lambda}})$ and $S_{\lambda}|_\mu=0$.  By the recurrence relation as in Proposition \ref{prop_mu_comp}, $\{S_{\lambda}|_\mu:\lambda \in \mathcal{P}(d,n)\}$ can be computed as a polynomial in $\{s_i(S_{\lambda}|_{s_i\mu}):\lambda \in \mathcal{P}(d,n)\}$ and $e_i(a)$.   Thus considering $\nu=s_i\eta$, we have $s_i(S_{\lambda}|_{s_i\mu})$ is a polynomial in $\frac{a_{\nu}}{a_{\mu}}$ for some $\nu\leq \mu$.  Also, $s_i\mu\prec\mu$ implies $i+1\in \{\widetilde{\mu}_1,\dots,\widetilde{\mu}_d\}$ but not $i$. Thus $e_i(a)=1-\frac{a_i}{a_{i+1}}$ can be written as $\frac{a_\nu}{a_\mu}$ for some $\nu\leq \mu$. 
\end{proof}

\begin{remark}\label{rem_pol_1}
For every $\lambda,\mu\in \mathcal{P}(d,n)$, there exist polynomials $f_{\lambda\mu}$ such that $$S_\lambda|_\mu=f_{\lambda\mu}\big((\frac{a_\nu}{a_\mu})_{\nu\preceq \mu}\big).$$ 
     % Let $\mu \succ \lambda$. If $\mu$ is obtained by attaching one box in $\lambda$, then $$S_\mu|_\mu=(1-\frac{a_{{\lambda}}}{a_{{\mu}}})S_\lambda|_\mu.$$
\end{remark}

\subsection{Factorial Grothendieck polynomials}
In this subsection, we discuss the factorial Grothendieck polynomial introduced in \cite{mc}, and study its combinatorial and geometric properties. We also define a divided difference operator on the algebra generated by factorial Grothendieck polynomials that behaves similarly to the divided difference operator on the Schubert classes. 

Define binary operators $\oplus$ and $\ominus$ by 
$$x\oplus y:=x+y- xy, ~~x\ominus y:=\frac{x-y}{1- y}.$$

We also define a deformation of $k$-th power of $x$ with parameters $b=(b_1,b_2,\dots)$ by
$$[x|b]^k:=(x\oplus b_1)(x\oplus b_2)\dots(x\oplus b_k).$$

Let $\lambda \in \mathcal{P}_d$ and $b=(b_1,b_2,\dots)$ be an infinite sequence. Define the following function. 
\begin{equation}\label{def_gro_pol}
   G_{\lambda}(x_1,\dots,x_d|b_1,b_2,\dots,)=\frac{\det([x_i|b]^{\lambda_j+d-j}(1-x_i)^{j-1})_{d\times d}}{\prod_{1\leq i<j\leq d}(x_i-x_j)}.
\end{equation}

 The determinant in the numerator on the right-hand side is a multiple of the denominator $\prod_{1\leq i<j\leq d}(x_i-x_j)$. The reason is that for any $1\leq k<\ell\leq d$, the $k$-th row and $\ell$-th row are the same if we substitute $x_k=x_\ell$. Thus  $G_{\lambda}(x_1,\dots,x_d|b_1,b_2,\dots)$ becomes a polynomial, and it is known as the factorial Grothendieck polynomial. 

\begin{example}\label{eg_fac_gro}
   $G_{(0)}(x|b)=1$. Let $\lambda=(1)\in \mathcal{P}_1$ then $G_{(1)}(x|b)=x_1\oplus b_1=x_1+b_1+x_1b_1=1-(1-x_1)(1-b_1)$. Similarly, if we consider $(1)=(1,0,\dots,0)\in \mathcal{P}_d$ then $G_{(1)}(x|b)=1-\xi(x)\xi(b)$, where $\xi(x)=\prod_{i=1}^d(1- x_i)$ and $\xi(b)=\prod_{i=1}^d(1- b_i)$. 
\end{example}

% \begin{proposition}
%     The factorial Grothendieck polynomials are symmetric in $x_1,x_2,\dots,x_d$.
% \end{proposition}

 Let $\mathcal{R}$ be the localization of $\ZZ[b_1,b_2,\dots]$ by the multiplicative system formed by the products of $\{1-b_i~|~i\geq 1\}$.  Identifying $b_i=1-a_i$ yields an isomorphism $\mathcal{R}\cong \mathbb{Z}[a_1^{\pm 1},a_2^{\pm 1},\dots ]$, and we obtain the following:

\begin{proposition}\cite{mc,ikNa}\label{cor_Gro_pol}
     $\{G_{\lambda}(x_1,\dots,x_d|1-a_1,1-a_2,\dots)\}_{\lambda \in \mathcal{P}_d}$ form a $\ZZ[a_1^{\pm 1},a_2^{\pm1},\dots]$-basis of $\ZZ[a_1^{\pm1},a_2^{\pm1},\dots][x_1,\dots,x_d]^{S_d}$. 
\end{proposition}

% \begin{proposition}\cite{mc,ikNa}\label{thm_basis_ord}
%     $\{G_{\lambda}(x_1,\dots,x_d|b_1,b_2,\dots)\}_{\lambda \in \mathcal{P}_d}$ form an $\mathcal{R}$-basis of $\mathcal{R}[x_1,\dots,x_d]^{S_d}$. 
% \end{proposition}

 If we set $b_i=0$ in $G_{\lambda}(x_1,\dots,x_d|b_1,b_2,\dots)$ for all $i$, then we call it by Grothendieck polynomial and denote it by $G_{\lambda}(x_1,\dots,x_d)$. Thus 
$$G_{\lambda}(x_1,\dots,x_d)=\frac{\det(x_i^{\lambda_j+d-j}(1-x_i)^{j-1})_{d\times d}}{\prod_{1\leq i<j\leq d}(x_i-x_j)}.$$

\begin{proposition}\cite{mc,ikNa}
     $\{G_{\lambda}(x_1,\dots,x_d)\}_{\lambda \in \mathcal{P}_d}$ form a $\ZZ$-basis of $\ZZ[x_1,\dots,x_d]^{S_d}$.
\end{proposition}

\begin{remark}
Let $\beta$ be a parameter. One can define a `$\beta$ deformed version' of the factorial Grothendieck polynomial by defining $x\oplus y=x+y+\beta xy$, $x\ominus y=\frac{x-y}{1+\beta y}$ and 
$${G}_{\lambda}^{(\beta)}(x_1,\dots,x_d|b_1,b_2,\dots):=\frac{\det([x_i|b]^{\lambda_j+d-j}(1+\beta x_i)^{j-1})_{d\times d}}{\prod_{1\leq i<j\leq d}(x_i-x_j)}.$$  
If we set $\beta=0$, then ${G}_{\lambda}^{(\beta)}(x_1,\dots,x_d|b_1,b_2,\dots)$ becomes factorial Schur polynomial as in \cite{MS}, if we further restrict $b_i=0$, then it becomes the Schur polynomial. If we specialize $\beta=-1$, then ${G}_{\lambda}^{(\beta)}(x_1,\dots,x_d|b_1,b_2,\dots)$ coincides with ${G}_{\lambda}(x_1,\dots,x_d|b_1,b_2,\dots)$ as in \eqref{def_gro_pol}. Since the case $\beta=-1$ is relevant to equivariant $K$-theory of Grassmannian, in this paper, we stick to $\beta=-1$ and work with ${G}_{\lambda}(x_1,\dots,x_d|b_1,b_2,\dots)$ as in \eqref{def_gro_pol}.
\end{remark}

%\subsection{Definition of Grothendieck polynomial using Divided difference operator}

We define the divided difference operator $\pi_i$ on $\mathcal{R}[x_1,\dots,x_d]^{S_d}$ by the following. For $i\geq 1$, $s_i$ acts on $\mathcal{R}[x_1,\dots,x_d]^{S_d}$ by interchanging $b_i$ and $b_{i+1}$. Let $f\in \mathcal{R}[x_1,\dots,x_d]^{S_d}$, define
$$\pi_if:=\frac{f-(1- e_i(b))s_if}{e_i(b)},$$
where $e_i(b)=b_i\ominus b_{i+1}$. Note that if we replace $b_i=1-a_i$ then $b_i\ominus b_{i+1}=1-\frac{a_i}{a_{i+1}}=e_i(a)$.

 For notational convenience, we denote $G_{\lambda}(x_1,\dots,x_d|b_1,b_2,\dots)$ by $G_\lambda(x|b)$ for $\lambda\in \mathcal{P}_d$. The next Proposition follows using the same argument as in \cite[Theorem 6.1]{ikNa}.
\begin{proposition}\label{prop_div_dif_ope}
    We have \[\pi_i G_\lambda(x|b)=\begin{cases}
              G_{s_i\lambda}(x|b) & \text{ if } s_i\lambda\prec\lambda\\
      G_\lambda(x|b) & \text{ if } s_i\lambda \succeq \lambda.
    \end{cases}\]
\end{proposition}
\begin{proof}
  If $s_i\lambda\succeq \lambda$, one see that $G_{\lambda}(x|b)$ is symmetric with respect to $b_i$ and $b_{i+1}$. Thus $s_iG_\lambda(x|b)=G_\lambda(x|b)$ and consequently $\pi_iG_\lambda(x|b)=G_\lambda(x|b)$.
  If $s_i\lambda\prec \lambda$ then $\tilde{\lambda}_j=i+1$ for some $j$.
 Since the permutation $s_i$ only interchange the variables $b_i$ and $b_{i+1}$ and do not affect the $x$ variables we focus on $[x|b]^{\lambda_\ell+d-\ell}=[x|b]^{\tilde{\lambda}_{d+1-\ell}-1}$. Now using
 $$x\oplus b_i-(1-b_i\ominus b_{i+1})(x\oplus b_{i+1})=b_i\ominus b_{i+1},$$
  we have $\pi_i(x\oplus b_i)=1$. Therefore, $\pi_i[x|b]^{\tilde{\lambda}_{j}-1}=\pi_i[x|b]^{i}=[x |b]^{i-1}$. This completes the proof.

 %=\begin{cases}G_{s_i\lambda}(x|b) & \text{ if } s_i\lambda\prec\lambda\\ G_\lambda(x|b) & \text{ if } s_i\lambda \succeq \lambda.  \end{cases}\]
\end{proof}

% \subsection{Algebraic localization map $\phi$}

% Let $\psi_{\mu}:\mathcal{R}[x_1,x_2,\dots,x_d]^{S_d}\to \mathcal{R}$ be the map defined by the substitution $x=b_\mu$, where $b_\mu=(b_{\mu_1}, b_{\mu_2},\dots, b_{\mu_d})$ and $b_{\mu_i}=\frac{-b_{\widetilde{\mu}_i}}{1+\beta b_{\widetilde{\mu}_i}}$.
% \begin{definition}
%     Define a homomorphism of $\mathcal{R}$-algebra map 
%     $$\Psi: \mathcal{R}[x_1,x_2,\dots,x_d]^{S_d}\to Fun(\mathcal{P}_{d,n}, \mathcal{R}), ~~~ f\mapsto (\mu \to \psi_{\mu}(f)).$$
% \end{definition}

% Let $\mathcal{X}$ be the subring of $Fun(\mathcal{P}_{d,n}, \mathcal{R})$ defined by $$\mathcal{X}=\{f=(f_\lambda)\in \bigoplus_{\lambda \in \mathcal{P}_{d,n}}\mathcal{R}: f_{\lambda}-f_{(i,j)\lambda}\in (b_i\ominus b_j){\cdot}\mathcal{R}\}.$$

% \begin{proposition}
%     $Im(\Psi)\subset \mathcal{X}$.
% \end{proposition}

\subsection{The vanishing property and algebraic localization map}
In this subsection, we describe that the factorial Grothendieck polynomial satisfies vanishing property and construct the algebraic localization map that proves the correspondence between factorial Grothendieck polynomials and Schubert classes.

 %The factorial Grothendieck polynomial $G_\lambda(x|1-a)$ carries the geometrical properties of the Schubert basis $S_\lambda\in K_{T^n}(\G(d,n))$ as in Proposition \ref{prop_van_pro_S}.
 Let $\mu=(\mu_1,\dots,\mu_d)\in \mathcal{P}_d$ and for all $1\leq i\leq d$. Define $\widetilde{\mu}_i:=\mu_{d+1-i}+i$ and $\Psi_\mu(G_\lambda(x|1-a))$ by replacing $x_i$ by $1-\frac{1}{a_{\widetilde{\mu}_i}}$ in $G_\lambda(x|1-a)$.
\begin{proposition}[vanishing property]\cite[Proposition 2.2]{ikNa}\label{prop_van_pro}
     \[\Psi_\mu(G_\lambda(x|1-a))=\begin{cases}0 &\text{ if } \lambda\npreceq\mu\\
\prod_{\nu\in R(\lambda)}(1-\frac{{a}_{\nu}}{{a}_{\lambda}}) &\text{ if } \lambda=\mu
\end{cases}\]
\end{proposition}

Consider a $\ZZ[a_1^{\pm1},a_2^{\pm1},\dots]$-algebra structure on $K_{T^n}(\G(d,n))$ through the map $$\ZZ[a_1^{\pm1},a_2^{\pm1},\dots] \to  \ZZ[a_1^{\pm1},\dots,a_n^{\pm1}],\text{ sending } a_i\to 1 \text{ for } i>n.$$ Now we consider two $\ZZ[a_1^{\pm1},a_2^{\pm1},\dots]$-algebras. The first is $\ZZ[a_1^{\pm1},a_2^{\pm1},\dots][x_1,\dots,x_d]^{S_d}$, which has the basis $\{G_{\lambda}(x|1-a):\lambda \in \mathcal{P}_d\}$. The second is $K_{T^n}(\G(d,n))$ with the Schubert basis $\{S_\lambda:\lambda \in \mathcal{P}(d,n)\}$. Also, $\mathcal{P}(d,n)\subset \mathcal{P}_d$.  We have the following algebra homomorphism:

\begin{theorem}\label{thm_sur_hom}
There exists a surjective homomorphism of $\ZZ[a_1^{\pm1},a_2^{\pm1},\dots]$-algebra
\[\Psi\colon\ZZ[a_1^{\pm1},a_2^{\pm1},\dots][x_1,\dots,x_d]^{S_d}\to K_{T^n}(\G(d,n))\]
which sends $G_{\lambda}(x|1-a)$ to the Schubert class $S_{\lambda}$ if $\lambda\in \mathcal{P}(d,n)$ and 0 otherwise.
 \end{theorem}
\begin{proof}
For every $\mu\in \mathcal{P}(d,n)$, define a $\ZZ[a_1^{\pm1},a_2^{\pm1},\dots]$-algebra homomorphism $$\Psi_{\mu}: \ZZ[a_1^{\pm1},a_2^{\pm 1},\dots][x_1,\dots,x_d]^{S_d}\to \ZZ[a_1^{\pm 1},\dots,a_n^{\pm1}]$$ by $\Psi_{\mu}(x_i)=1-\frac{1}{a_{\widetilde{\mu}_i}}$, $\Psi_{\mu}(a_i)=a_i$ for $1\leq i\leq n$, and  $\Psi_{\mu}(a_i)=1$ for $i> n$. 

This map is well defined by the vanishing property of $G_\lambda(x|1-a)$, as in Proposition \ref{prop_van_pro}. More explicitly, let $F= \sum_{\lambda\in \mathcal{P}_d}k_\lambda G_\lambda(x|1-a)$ be an arbitrary element in $\ZZ[a_1^{\pm1},a_2^{\pm1},\dots][x_1,\dots,x_d]^{S_d}$. For each $\mu$, the value $\Psi_\mu(G_\lambda(x|1-a))$ is non-zero for only finitely many $\lambda$ such that $\lambda\preceq \mu$. Therefore, $\Psi_\mu(F)$ is well defined element in $\ZZ[a_1^{\pm 1},\dots,a_n^{\pm1}]$.
Define the algebraic localization map
$$\Psi\colon \ZZ[{a}^{\pm1}_1,{a}^{\pm1}_2,\dots][x_1,\dots,x_d]^{S_d}\to \bigoplus_{\mu\in \mathcal{P}(d,n)} \ZZ[{a}_1^{\pm1}, \dots, {a}_n^{\pm1}]$$
by $(\Psi(f))_\mu=\Psi_\mu(f)$
for $\mu\in\mathcal{P}(d,n)$. 

We claim that ${\rm{Image}}(\Psi)\subset K_{T^n}(\G(d,n))$. Let $F= \sum_{\lambda\in \mathcal{P}_d}k_\lambda G_\lambda(x|1-a)$ be an arbitrary element in $\ZZ[a_1^{\pm1},a_2^{\pm1},\dots][x_1,\dots,x_d]^{S_d}$. Let $\eta$ and $\mu$ be two Schubert symbol such that $\eta=(i,j)\mu$. Since $F$ is symmetric, we focus on only two variables $a_i$ and $a_j$. By construction, substituting $a_i$ by $a_j$ in $\Psi_\mu(F)$ yields $\Psi_\eta(F)$. Thus the difference $\Psi_\eta(F)-\Psi_\mu(F)$ is divisible by $(1-\frac{a_i}{a_j})=(1-\frac{a_\mu}{a_\eta})$. Therefore $\Psi(F)\in K_{T^n}(\G(d,n))$. 

 Moreover, $\Psi\circ\pi_i(f)=\pi_i\circ\Psi(f)$ for every $1\leq i\leq n-1$ using the same argument as in \cite[Proposition 7.4]{IMN}. Now using Proposition \ref{prop_div_dif_ope}, Definition \ref{def:sch bas} and the uniqueness of Schubert classes implies that $\Psi(G_{\lambda}(x|1-a))$ coincide with the Schubert class $S_\lambda$ if $\lambda\in \mathcal{P}(d,n)$. If $\lambda\in \mathcal{P}_d\setminus\mathcal{P}(d,n)$ and $\mu\in \mathcal{P}(d,n)$ then $\lambda\npreceq \mu$ and $\Psi_\mu(G_\lambda(x|1-a))=0$ follows from the Proposition \ref{prop_van_pro}. Thus $\Psi(G_\lambda(x|1-a))=0$ if $\lambda\in \mathcal{P}_d\setminus\mathcal{P}(d,n)$.
\end{proof}

\begin{corollary}\label{cor_alg_mor_gsm} 
There exist algebra homomorphisms 
$$\Psi_{\mu}: \ZZ[{a}^{\pm1}_1,{a}^{\pm1}_2,\dots][x_1,\dots,x_d]^{S_d}\to \ZZ[{{a}}_1^{\pm1}, \dots, { a}_n^{\pm1}]$$  such that $\Psi_{\mu}(G_{\lambda}(x|1-a))=S_{\lambda}|_{\mu}$ for every $\lambda,\mu \in \mathcal{P}(d,n)$.
%In other words, the restriction of the Schubert class $S_\lambda\in K_{T^n}(\G(d,n))$ in the torus fixed points can be explicitly computed as the images of the factorial Grothendieck polynomial.
\end{corollary}
\begin{proof}
If $\lambda\in \mathcal{P}(d,n)$ then using Theorem \ref{thm_sur_hom}, we have $\Psi(G_\lambda(x|1-a))=S_\lambda$.
Comparing $\mu$-th component both side $\Psi_{\mu}(G_{\lambda}(x|1-a))=S_{\lambda}|_{\mu}$ for every $\mu\in \mathcal{P}(d,n)$.
\end{proof}
\begin{example}\label{eg_sur_hom}
Using example \ref{eg_fac_gro}, we have $G_{(1)}(x|1-a)=1-(a_1\cdots a_d)\prod_{i=1}^d(1-x_i)$. For $\mu\in \mathcal{P}(d,n)$, if we replace $x_i$ by $1-\frac{1}{a_{\widetilde{\mu}_i}}$ in $G_{(1)}(x|1-a)$ then $$\Psi_\mu(G_{(1)}(x|1-a))=1-\frac{a_{(0)}}{a_\mu}=S_{(1)}|_\mu.$$     
\end{example}

% \begin{lemma}\label{lem_st_con}
%    The structure constant $\mathscr{C}_{\lambda \mu}^{\nu}$ has the following properties.
% 	\begin{enumerate}
% 		\item ${\mathscr{C}}_{\lambda \mu}^{\nu}=0$ unless  $\nu \succeq \lambda, \mu$.   
% 		\item If $\lambda=\nu$ we have ${\mathscr{C}}_{\lambda\mu}^{\lambda}={{S}_{\mu}}|_{\lambda}$.
% 	\end{enumerate}
% \end{lemma}

Next, we discuss the Chevalley rule in $K_{T^n}(\G(d,n))$. For every partition $\lambda\in \mathcal{P}(d,n)$, we associate its Young diagram by left aligning $d$ rows of boxes, where $i$-th row contains $\lambda_i$ boxes. For two partitions $\lambda,\mu \in \mathcal{P}(d,n)$ with $\lambda\preceq \mu$, we denote $\lambda\Rightarrow \mu$ if every box in the skew diagram $\mu\setminus\lambda$ lies in distinct row and distinct column. The notation $\lambda\Rightarrow \mu$ also allows the case $\mu=\lambda$. If we want to exclude the possibility that $\mu=\lambda$, then we write  $\lambda\Rightarrow^{*} \mu$.

\begin{proposition}[Chevalley rule]\label{thm_eq_ch}
  \begin{equation*}
    S_{\lambda}S_{(1)}=(1-\frac{{a}_{(0)}}{a_{\lambda}})S_\lambda+\frac{{a}_{(0)}}{a_{\lambda}}\sum_{\mu:\lambda\Rightarrow^{*} \mu}(-1)^{|\mu\setminus \lambda|{-1}}S_{\mu}.
\end{equation*}  
\end{proposition}
\begin{proof}
     Using \cite[Proposition 4.8]{mc}, we have 
     \begin{equation}\label{eq_che_rul}
         \xi(x)G_{\lambda}(x|1-{{a}})=\frac{1}{{a}_{\lambda}}\sum_{\mu:\lambda \Rightarrow\mu}(-1)^{|\mu\setminus \lambda|}G_{\mu}(x|1-{{a}}).
     \end{equation}
Now $G_{(1)}(x|1-a)=1-a_{(0)}\xi(x)$. This implies $\xi(x)=\frac{(1-G_{(1)}(x|1-a))}{a_{(0)}}.$
Thus \begin{align*}
  &(1-G_{(1)}(x|1-a))G_{\lambda}(x|1-{{a}})=\frac{a_{(0)}}{{a}_{\lambda}}\sum_{\mu:\lambda \Rightarrow\mu}(-1)^{|\mu\setminus \lambda|}G_{\mu}(x|1-{{a}}). \text{ Therefore,} \\
 &G_{(1)}(x|1-a)G_{\lambda}(x|1-{{a}})=(1-\frac{a_{(0)}}{{a}_{\lambda}})G_{\lambda}(x|1-{{a}})+\frac{{a}_{(0)}}{a_{\lambda}}\sum_{\mu:\lambda\Rightarrow^{*} \mu}(-1)^{|\mu\setminus \lambda|{-1}}G_{\mu}(x|1-{{a}}).
\end{align*}
Now the proof follows using the Theorem \ref{thm_sur_hom}.
\end{proof}

Let $\varepsilon: \ZZ[a_1^{\pm1},\dots,a_n^{\pm1}]\to \ZZ$ be the ring homomorphism given by $\varepsilon(a_i)=1$. Equivalently $\varepsilon(1-a_i)=0$. The ordinary $K$-theory $K(\mathrm{Gr}(d,n))$ is an algebra over $\mathbb{Z}$, and can be recovered from the equivariant theory by base change:
$$K(\mathrm{Gr}(d,n)) \cong K_T(\mathrm{Gr}(d,n)) \otimes_{R(T)} \mathbb{Z}.$$ The map $\varepsilon$ corresponds to trivializing the torus action.
Thus, we have the forgetful map 
$$ K_T(\mathrm{Gr}(d,n)) \longrightarrow K(\mathrm{Gr}(d,n)).$$
It takes a $T$-equivariant vector bundle and simply forgets the torus action, yielding the underlying bundle in the ordinary category. Let $\mathbb{S}_\lambda$ be the image of $S_\lambda$ through the forgetful map.
\begin{proposition}
       There exists a surjective homomorphism of $\ZZ$-algebra 
    \[ \ZZ[x_1,x_2,\dots,x_d]^{S_d}\to K(\G(d,n))\] which send $G_{\lambda}(x_1,\dots,x_d)$ to $\mathbb{S}_\lambda$ if $\lambda\in \mathcal{P}(d,n)$ and 0 otherwise. 
\end{proposition}

The Chevalley rule in $K(\G(d,n))$ follows as a consequence of Proposition \ref{thm_eq_ch}.
\begin{corollary}[Chevalley rule]\cite[Theorem 3.1]{Len}
  \begin{equation*}
\mathbb{S}_{\lambda}\mathbb{S}_{(1)}=\sum_{\mu:\lambda\Rightarrow^{*} \mu}(-1)^{|\mu\setminus \lambda|{-1}}\mathbb{S}_{\mu}.
\end{equation*}      
\end{corollary}

\section{Schubert classes in the equivariant $K$-theory of {\Pl} coordinates}\label{sec_sc_cal_pldn}
In this section, we recall {\Pl} coordinates from Section \ref{subsec_pl_wgt_vec}, and describe a presentation of the $T^{n+1}$-equivariant $K$-theory of the {\Pl} coordinates. We discuss an isomorphism between $K_{T^{n+1}}(Pl(d,n))$ and $K_{T^n}(\G(d,n))$ and describe a combinatorial description of Schubert classes in $K_{T^{n+1}}(Pl(d,n))$.

 We recall the coordinate vector $f_{\lambda}$ in $\Lambda^d(\CC^n)$ described in Section \ref{subsec_pl_wgt_vec}. Define $$F(\lambda):=\pi^{-1}(f_\lambda)=\CC^*{\cdot}f_{\lambda}\subset Pl(d,n).$$
 
 The $T^{n+1}$ action on $Pl(d,n)$ defined in \eqref{eq_c_action}. This restricts to an action of $T^{n+1}$ in $F(\lambda)$ given by $$(t_1,t_2,\dots,t_n,t){\cdot}z=t{\cdot}t_{\lambda}{\cdot}z.$$
Therefore, $\widetilde{T}_\lambda:=\{(t_1,t_2,\dots,t_n,t)\in T^{n+1}:t{\cdot}t_{\lambda}=1\}$ acts trivially on $F(\lambda)$. We denote $R(T^{n+1})=\ZZ[a_1^{\pm1}, \dots, a_n^{\pm1}, u^{\pm1}]$. Then $$K_{T^{n+1}}( F(\lambda))=K_{\widetilde{T}_\lambda}(\{pt\})=R(\widetilde{T}_\lambda)=\frac{\ZZ[a_1^{\pm1},\dots, a_n^{\pm1}, u^{\pm1}]}{(1-ua_{\lambda})}.$$

We define $\ell(\lambda):=d(n-d)-|\lambda|$.
\begin{proposition}\label{thm:q-cw_srtucture}
 There exist a $T^{n+1}$-invariant stratification 
$\{X_{\lambda}\}_{\lambda\in \mathcal{P}(d,n)}$ such that for all $\lambda\in \mathcal{P}(d,n)$, the quotient $X_\lambda/\cup_{\mu\succ\lambda}X_\mu$
is homeomorphic to the Thom space $Th(\xi^{\lambda})$
of an  $T^{n+1}$-vector bundle 
\begin{equation}\label{eq_orb_vector_bdl_over_pt}
\xi^{\lambda} \colon E(\lambda)\times F(\lambda) \to F(\lambda).
\end{equation} 
\end{proposition}
\begin{proof}
The CW complex structure of $\G(d,n)$ described in Section \ref{subsec_q_cell_stu}, induces a $T^{n+1}$-invariant stratification $\{X_{\lambda}\}_{\lambda\in \mathcal{P}(d,n)}$ on $Pl(d,n)$ where, $$X_\lambda:=\cup_{\mu\succeq\lambda}\pi^{-1}(E(\mu))
%=\cup_{\mu\succeq\lambda}(E(\mu)\times\CC^*) .
$$   
Note that $\pi^{-1}(E(\mu))$ is homeomorphic to $E(\mu)\times F(\mu)$. Thus  $X_\lambda/\cup_{\mu\succ\lambda}X_\mu$ is the Thom space of the $T^{n+1}$-vector bundle $$E(\lambda)\times F(\lambda) \to F(\lambda).$$ 
This completes the proof.
%for all $\lambda\in \mathcal{P}(d,n)$.
%It remains to note that $E(\lambda)$ is $T^{n+1}$-equivariantly homeomorphic to $\CC^{\ell(\lambda)}$. 
\end{proof}

 A pair $(s,s')$ is called an inversion of $\lambda\in I(d,n)$ if $s\in \lambda, s'\notin \lambda$ and $s<s'$. The set of all inversions of $\lambda$ is denoted by $\rm{inv}(\lambda)$.
Now corresponding to $\text{inv}(\lambda)$, one can define a subset of Schubert symbols as follows 
\begin{equation}\label{eq_rev_lbd}
	I(\lambda):=\{\mu~|~\mu=(s,s')\lambda \text{ for } (s,s')\in \text{inv}(\lambda)\}.
\end{equation}
 Then the cardinality of the set $I(\lambda)$ is $\ell(\lambda)$ for every $\lambda\in \mathcal{P}(d,n)$. %where $\ell(\lambda^i)$ is the length of $\lambda^i$.
Note that the bundle in \eqref{eq_orb_vector_bdl_over_pt} is also an $T^{n+1}$-vector bundle.

%It seems better to give a notation for $F(\lambda){\cdot}f_\lambda$

\begin{proposition}\label{prop:cond_hhh}
The $T^{n+1}$-bundle in \eqref{eq_orb_vector_bdl_over_pt} has a decomposition $$\xi^{\lambda} \colon E(\lambda)\times F(\lambda) \to F(\lambda) \cong \bigoplus_{\mu: \mu\in I(\lambda)}(\xi^{\lambda\mu} \colon E_{\lambda\mu}\times F(\lambda)\to F(\lambda) ).$$ 
\end{proposition}
\begin{proof}
	Observe that $X_\lambda\setminus\cup_{\mu\succ\lambda}X_\mu=\pi^{-1}(E(\lambda))\cong E(\lambda)\times F(\lambda)$. Since $T^{n+1}$ is abelian, the $T^{n+1}$ action on $E(\lambda)\cong \CC^{\ell(\lambda)}$ determines the following decomposition
	$$E(\lambda)\cong \bigoplus_{\mu: \mu \in I(\lambda)}E_{\lambda\mu}$$ for some irreducible representation $E_{\lambda\mu}\cong \CC$ of $T^{n+1}$. Hence, the proof follows.	
\end{proof}

\begin{remark}\label{rem_eq_strati}
\begin{enumerate}
\item The attaching map $\eta_{\lambda} \colon S(\xi^\lambda) \to \cup_{\mu\succ\lambda}X_\mu$ for the vector bundle in \eqref{eq_orb_vector_bdl_over_pt} satisfies $\eta_\lambda|_{S(\xi^{\lambda\mu})}=h_{\lambda\mu}\circ\xi^{\lambda\mu}$ where $h_{\lambda\mu}:F(\lambda)\to F(\mu)$ is given by $h_{\lambda\mu}(t{\cdot}f_{\lambda})=(t{\cdot}f_{\mu})$. The induced map in the equivariant $K$-theory $$h_{\lambda\mu}^*\colon K_{T^{n+1}}(F(\mu))\to K_{T^{n+1}}(F(\lambda))$$ is 
%$h_{\lambda\mu}^*\colon \frac{\ZZ[a_1^{\pm},a_2^{\pm}, \dots, a_n^{\pm}, u^{\pm}]}{(1-u{\cdot}a_{\mu})}\to \frac{\ZZ[a_1^{\pm},a_2^{\pm}, \dots, a_n^{\pm},u^{\pm}]}{(1-u{\cdot}a_{\lambda})}$
defined by $$h_{\lambda\mu}^*(x+(1-ua_{\mu}))=x+(1-ua_{\lambda}).$$

\item The $T^{n+1}$-action on $Pl(d,n)$ restricts to a $T^{n+1}$-action on $E_{\lambda\mu}\times F(\lambda)$ by $$(t_1,t_2,\dots,t_n,t)(z_1,z_2)=(t{\cdot}t_\mu z_1, t{\cdot}t_\lambda z_2).$$ Therefore, the equivariant Euler classes $\{e_{T^n}(\xi^{\lambda\mu})~|~\mu\in I(\lambda)\}$ of the bundles $\xi^{\lambda\mu}$ are given by $1-ua_{\mu}$. They are pairwise prime by \cite[Lemma 5.2]{HHH}.
\end{enumerate}
\end{remark}

\begin{theorem}\label{thm_eq_K_thm_pldn}
    The $T^{n+1}$-equivariant $K$-theory ring of $Pl(d,n)$ is given by  $$K_{T^{n+1}}(Pl(d,n))=\Big\{f\in \bigoplus_{\lambda\in \mathcal{P}(d,n)} R(\widetilde{T}_\lambda)~~|~~ f|_\lambda-h_{\lambda\mu}^*(f|_\mu) \text{ is divisible by } 1-ua_{\mu} \text{ if } \mu=I(\lambda)\Big\}.$$
\end{theorem}
\begin{proof}
    This follows from \cite[Theorem 3.1]{HHH}.
\end{proof}
\begin{theorem}\label{thm_iso_alg}
    The projection map $\pi \colon Pl(d,n)\to \G(d,n) $ induces the isomorphism $$\pi^*\colon K_{T^n}(\G(d,n)) \to K_{T^{n+1}}(Pl(d,n))$$
    as a $R(T^n)$-algebra, where the map $R(T^n)$ to $R(T^{n+1})$ is given by the inclusion.
\end{theorem}
 \begin{proof}
The circle subgroup $\{(0,\dots,0,t):t\in S^1\}$ of $T^{n+1}$ acts freely on $Pl(d,n)$. Then, we have the above isomorphism using \cite[Proposition 2.1]{Segal}.  
 \end{proof}

Define 
\begin{equation}\label{eq_pl_sc_bas}
    pS_{\lambda}:=\pi^*(S_{\lambda}).
\end{equation}

\begin{lemma}\label{lem_pslam}
For every $\lambda\in \mathcal{P}(d,n)$, the element $pS_{\lambda} \in K_{T^{n+1}}(Pl(d,n))$ satisfies the following conditions:
\begin{enumerate}
    \item $pS_{\lambda}|_\mu=0\text{ if}~  \mu\nsucceq\lambda$\\
    \item $pS_{\lambda}|_\lambda=\prod_{\nu \in R(\lambda)}(1-{u{ a}_{\nu}})$.   
\end{enumerate}
\end{lemma}
\begin{proof}
The proof follows from the definition of $pS_{\lambda}$ as in \eqref{eq_pl_sc_bas} together with Proposition \ref{prop_van_pro_S}. 
    Note that $1-\frac{{ a}_{\nu}}{{a}_{\lambda}}=1-{u{a}_{\nu}}\in \frac{\ZZ[a_1^{\pm1}, \dots, a_n^{\pm1},u^{\pm1}]}{(1-ua_{\lambda})}$.
\end{proof}
%Recall $(0)$ denotes $(0,0,\dots,0)\in \mathcal{P}(d,n)$ and $(1)$ denotes $(1,0,\dots,0)\in \mathcal{P}(d,n)$.

\begin{lemma}\label{lem_div class}
    $pS_{(1)}|_{\lambda}=1-ua_{{(0)}}$.
\end{lemma}
\begin{proof}
 We have $S_{(1)}|_{\lambda}=1-\frac{a_{(0)}}{a_\lambda}$ by Lemma \ref{lem_div_case} and $1-\frac{a_{(0)}}{a_\lambda}=1-ua_{(0)}\in \frac{\ZZ[a_1^{\pm1}, \dots, a_n^{\pm1},u^{\pm1}]}{(1-ua_{\lambda})}.$
\end{proof}

\begin{remark}\label{rem_pol_2}
  There exist polynomials $f_{\lambda\mu}$ such that  $pS_\lambda|_\mu=f_{\lambda\mu}((ua_\nu)_{\nu\preceq\mu})$, for every $\lambda,\mu \in \mathcal{P}(d,n)$. This follows from Remark \ref{rem_pol_1} and Definition of $pS_\lambda$ as in \eqref{eq_pl_sc_bas}.
\end{remark}

% \begin{lemma}
%     The isomorphism $\pi^*$ sends the element $\frac{1-S_{(1)}}{a_{(0)}}\in K_{T^n}(\G(d,n))$ to $u{\cdot}1\in K_{T^{n+1}}(Pl(d,n))$.
% \end{lemma}
% \begin{proof}
% Using Lemma \ref{lem_div class}, $u{\cdot}1=\frac{1-pS_{(1)}}{a_{(0)}}=\pi^*(\frac{1-S_{(1)}}{a_{(0)}})$. 
% \end{proof} 

% We recall $G_{(1)}(x|b)=1-\xi(x)\xi(b)$, where $\xi(x)=\prod_{i=1}^d(1- x_i)$. Note that $\xi(x)$ is a symmetric polynomial in $x_1,\dots,x_d$. Thus $\xi(1-a)=a_1\dots a_d=a_{(0)}$ and $G_{(1)}(x|1-a)=1-\xi(x)a_{(0)}$. Note that $\Psi(G_{(1)}(x|1-a))=S_{(1)}$. Thus 

\begin{proposition}\label{sur_alg_hom_pldn}
 There exists a surjective homomorphism of $\ZZ[a_1^{\pm1},a_2^{\pm1},\dots]$-algebra
 $$\phi:\ZZ[a_1^{\pm1},a_2^{\pm1},\dots][x_1,\dots,x_d]^{S_d}\to K_{T^{n+1}}(Pl(d,n))$$ 
%of $\ZZ[a_1^{\pm1},a_2^{\pm1},\dots]\to \ZZ[a_1^{\pm1},\dots,a_n^{\pm 1}]$ algebra sending $a_i\to 1$ for $i>n$,
such that
$\phi(G_\lambda(x|1-a))=pS_\lambda$  if $\lambda\in \mathcal{P}(d,n)$ and  
0 otherwise.
\end{proposition}

\begin{proof}
    Define $\phi:=\pi^*{\circ}\Psi$. Then the proof follows by Theorem \ref{thm_sur_hom} and Theorem \ref{thm_iso_alg}.
\end{proof}

\begin{lemma}
    The algebra homomorphism $\phi$ sends the symmetric polynomial $\xi(x)=\prod_{i=1}^d(1-x_i)$ to $u{\cdot}1\in K_{T^{n+1}}(Pl(d,n))$.
\end{lemma}
\begin{proof}
  Using Example \ref{eg_sur_hom}, it follows that $$\Psi(\xi(x))=\Psi\Big(\frac{1-G_{(1)}(x|1-a)}{a_{(0)}}\Big)=\frac{1-S_{(1)}}{a_{(0)}}. \text{  Thus}$$ $$(\pi^*\circ\Psi)(\xi(x))=\pi^*\Big(\frac{1-S_{(1)}}{a_{(0)}}\Big)=\frac{1-pS_{(1)}}{a_{(0)}}=u{\cdot}1.$$
  The last equality follows from  Lemma \ref{lem_div class}.
\end{proof}

Now $\ZZ[a_1^{\pm1},a_2^{\pm1},\dots][x_1,\dots,x_d]^{S_d}$ also has a $\ZZ[\xi(x)^{\pm1},a_1^{\pm1},a_2^{\pm1},\dots]$-algebra structure by obvious multiplication and $K_{T^{n+1}}(Pl(d,n))$ can be considered as a $\ZZ[u^{\pm1},a_1^{\pm1},\dots,a_n^{\pm1}]$-algebra. Thus the map $\phi$ in Proposition \ref{sur_alg_hom_pldn} is also an algebra homomorphism 
   %$$\phi:\ZZ[{a}^{\pm1}_1,{a}^{\pm1}_2,\dots][x_1,\dots,x_d]^{S_d}\to K_{T^{n+1}}(Pl(d,n))$$
 with respect to $\ZZ[\xi(x)^{\pm1},a_1^{\pm1},a_2^{\pm1},\dots] \to \ZZ[u^{\pm1},a_1^{\pm1},\dots,a_n^{\pm1}]$ algebra, where $\xi(x)\to u$ and $a_i\to 1$ for $i>n$.

\begin{remark}\label{rmk_phi_mu}
    For every $\mu\in \mathcal{P}(d,n)$ there exist maps $$\phi_\mu:\ZZ[a_1^{\pm1},a_2^{\pm1},\dots][x_1,\dots,x_d]^{S_d}\to\frac{\ZZ[a_1^{\pm1}, \dots, a_n^{\pm1}, u^{\pm1}]}{(1-ua_{\mu})}$$ such that
$\phi=(\phi_\mu)_{\mu\in \mathcal{P}(d,n)}$ and $\phi_\mu(G_{\lambda}(x|1-a))=pS_{\lambda}|_\mu$. Moreover, $\phi_\mu(\xi(x))=\frac{1}{a_\mu}$ and $\phi_\mu(a_i)=1$ for $i> n$.
\end{remark}

\section{Schubert classes in the equivariant $K$-theory of divisive weighted Grassmann orbifolds}\label{Sec_Sc_bas_eq_K_th_wgt_gsm}
In this section, we explore the equivariant $K$-theory of the divisive weighted Grassmann orbifolds. We provide a combinatorial description of the Schubert classes and explore that they form a basis in the equivariant $K$-theory of divisive weighted Grassmann orbifolds. 
    
    %The $T^{n+1}$ action on $Pl(d,n)$ described in \eqref{eq_c_action} is induced by the injective map $$\rho\colon T^{n+1} \to (S^1)^{n \choose d}$$defined by $\rho(t_1, t_2, \dots, t_n,t) =(t \cdot t_{\lambda})_{\lambda\in \mathcal{P}(d,n)}$, where $t_\lambda=t_{\widetilde{\lambda}_1}t_{\widetilde{\lambda}_2}\cdots t_{\widetilde{\lambda}_d}$.

% Note that $\rho$ is an injective group homomorphism. If $\rho(t_1, t_2, \dots, t_n,t) =(1,1,\dots,1)$ then $t{\cdot}t_{\lambda}=t{\cdot}t_{\mu}$ for any two Schubert symbol $\lambda$ and $\mu$.

% for any arbitrary $i,j$ with $i\neq j$ consider two partitions $\lambda$ and $\mu$ such that $\lambda=(i,j)\mu$.  Then $t{\cdot}t_{\lambda}=t{\cdot}t_{\mu}$ implies $t_i=t_j$.

% Consider a circle subgroup $S^1_{\bf c}$ of $\rho(T^{n+1})$ defined by
% $$S^1_{\bf c}:=\{(t^{c_0},\dots, t^{c_m}): t\in S^1\},$$ where ${\bf c}$ is a {\Pl} weight vector. Then $S^1_{\bf c}=\rho(WD)$ for some subgroup $WD$ of $T^{n+1}$ as defined in \eqref{eq_wd}. Define $T_{\bf c}:=\frac{\rho(T^{n+1})}{S^1_{\bf c}}$. Then $T_{\bf c}$ is isomorphic to $\frac{T^{n+1}}{WD}$ as $\rho$ is an injective map. 
%Consider the $T_{\bf c}$ action on $\G_{\bf c}(d,n)$. 

%One  circle subgroup $S^1_{\bf c}$ can corresponds to many subgroup $WD$ of $T^{n+1}$ as described in Remark \ref{lem_many_WD}.
%This helps to find a nice presentation of $R(T_{\bf c})$. 

\begin{lemma}\label{lem_div_a,1}
Every divisive weighted Grassmann orbifold $\G_{\bf c}(d,n)$ is homeomorphic to some divisive weighted Grassmann orbifold $\G_{\bf c'}(d,n)$ such that ${\bf c'}$ corresponds to $(W,a)$ for $W=(w_1,\dots,w_n)\in (\ZZ_{\geq0})^n$ and $a=1$.
\end{lemma}
\begin{proof}
Let $\G_{\bf c}(d,n)$ be a divisive weighted Grassmann orbifold. Then $c_\mu$ divides $c_\lambda$ for $\lambda\leq \mu$. Weighted Grassmann orbifolds $\G_{\bf c}(d,n)$ and $\G_{r\bf c}(d,n)$ are homeomorphic for every positive integer $r$, see \cite[Lemma 3.7]{br}. Thus, we can assume $c_\lambda=1$ for the partition $\lambda=(n-d)^d\in \mathcal{P}(d,n)$ of maximal length. The corresponding element in $I(d,n)$ is given by $(n-d+1,\dots,n)$. For every $i=n-d,\dots,n$, consider the pairs $(\lambda^i,\mu^i)$ of Schubert symbol, where $\lambda^{i}:=(n-2d+1,\dots,n-d+1, i)\in I(d,n)$, and $\mu^{i}:=(n-d,\dots,\hat{i},\dots,n)\in I(d,n)$ corresponds to two order sequences $n-2d+1<\dots<n-d-1$ and $n-d<\dots<n$ as in Definition \ref{def_plu_vec}. Therefore, using the Remark \ref{rmk_div_case}, we get $c_{\mu^{i}}=1$. Thus $w_n=w_i$ for all $i=n-d,\dots,n$. Moreover, $c_{\mu^{n-d}}=1$ implies that $dw_n+a=1$. Thus $a\cong 1(\text{mod}(d))$.
%For the integer solution of $w_n$, we choose $a=1$ and $w_n=0$. 
Now, the proof follows using the similar argument of \cite[Proposition 2.9]{BS}. Since ${\bf c}$ is divisive we have $w_i\geq w_{i+1}$ for all $i=1,\dots,n-1$. Hence the proof follows.
\end{proof}

\begin{remark}\label{rmk_div_case}
    Let $n_1,n_2,n_3$ be positive integers such that $n_1+1=n_2+n_3$. If $n_2$, $n_3$ divides $n_1$ then $\{n_1,1\}=\{n_2,n_3\}$.  To prove this, let $n_1=k_2n_2=k_3n_3$. Then $(k_2-1)n_2+1=n_3$ and $(k_3-1)n_3+1=n_2$. Then $(k_2-1)(k_3-1)n_2+k_3=n_2$. Thus either $k_2=1$ or, $k_3=1$. Thus either $n_2=1$ or, $n_3=1$.
\end{remark}

% \begin{example}
%     Consider a divisive {\Pl} weight vector  ${\bf c}=(12,12,12,3,3,3)$. Then $\G_{\bf c}(2,4)$ is homeomorphic to $\G_{\bf c'}(2,4)$, where ${\bf c'}=(4,4,4,1,1,1)$. Now ${\bf c'}$ correspond to $W=(3,0,0,0)$ and $a=1$ as explained in Lemma \ref{lem_div_a,1}. Note that ${\bf c}$ also corresponds to $W=(9,1,1,1)$ and $a=2$.
% \end{example}

Using Lemma \ref{lem_div_a,1}, any divisive \Pl~ weight vector ${\bf c}$ always corresponds to a pair $(W,a)$ such that $a=1$. Recall the $T^{n+1}$-action on $Pl(d,n)$ described in \eqref{eq_c_action} and the subgroup $WD$ of $T^{n+1}$ as defined in \eqref{eq_wd}. Thus we have the $T_{\bf c}:=\frac{T^{n+1}}{WD}$ action on $\G_{\bf c}(d,n)$.
Now, we prove that $T_{\bf c}$ is isomorphic with $T^n$. Define $\bar{f}\colon T^{n+1}\to T^n$ by 
$$\bar{f}(t_1,\dots,t_{n},t)=(t_1t^{-w_1},t_2t^{-w_2},  \dots, t_nt^{-w_n}).$$

For $t\in \CC^*$, $\bar{f}(t^{w_1}\dots,t^{w_n},t)=(1,\dots,1)$. Thus, $WD\subset \ker(\bar{f})$. Also, $$\ker(\bar{f})=\{(t_1,t_2,\dots,t_n,t):t_it^{-w_i}=1, \text{ for }1\le i\le n\}=\{(t^{w_1}\dots,t^{w_n},t):t\in \CC^*\}.$$ The inverse image of $(t_1,\dots,t_n)$ is $(t_1,t_2,\dots,t_n,1)$. Thus $\bar{f}$ is onto and it induces an isomorphism $$f\colon T_{\bf c}\to T^n.$$

 The inverse of $f$ is defined by $$f^{-1}: T^n\to T_{\bf c};\quad f^{-1}(t_1,\dots,t_n)=(t_1,\dots,t_n,1)+WD.$$ 
The isomorphism $f$ between $T^n$ and $T_{\bf c}$ induces a $T^n$-action on $\G_{\bf c}(d,n)$ by the following.
$$(t_1,\dots,t_n){\cdot}[\zeta_\lambda]=[t_{\lambda}\zeta_\lambda],$$
where $t_\lambda=t_{\tilde{\lambda}_1}t_{\tilde{\lambda}_2}\dots t_{\tilde{\lambda}_d}$. Consequently, $f$ induces an isomorphism between $K_{T^n}(\G_{\bf c}(d,n))$ and $K_{T_{\bf c}}(\G_{\bf c}(d,n))$.

The equivariant $K$ theory $K_{T_{\bf c}}(\G_{\bf c}(d,n))$ has a $R(T_{\bf c})$ algebra structure. Using the isomorphism $f$, we can write $R(T_{\bf c})=\ZZ[{\bf{a}}_1^{\pm1}, \dots, {\bf a}_n^{\pm1}]$, where ${\bf a}_i=\frac{a_i}{u^{w_i}}$. Thus, $R(T_{\bf c})$ can be regarded as a subring of $R(T^{n+1})$  by ${\bf a}_i=\frac{a_i}{u^{w_i}}\in \ZZ[a_1^{\pm1},\dots,a_n^{\pm1},u^{\pm1}]$. For $\lambda\in \mathcal{P}(d,n)$ denote ${\bf a}_{\lambda}:={\bf a}_{\widetilde\lambda_1}{\bf a}_{\widetilde\lambda_2}{\cdots}{\bf a}_{\widetilde\lambda_d}\in R(T_{\bf c})$. The next theorem follows by applying the same argument use in \cite[Theorem 4.7]{BS}. 
\begin{theorem}\label{thm_K_T_CGRBDN}
    The $T_{\bf c}$-equivariant $K$-theory of $\G_{\bf c}(d,n)$ is given by 
   {\small $$K_{T_{\bf c}}(\G_{\bf c}(d,n))=\Big\{f\in \bigoplus_{\lambda\in \mathcal{P}(d,n)} \ZZ[{\bf a}_1^{\pm1}, \dots, {\bf a}_n^{\pm1}]~~|~~ f|_\lambda-f|_\mu \text{ is divisible by } (1-\frac{{\bf a}_{\lambda}}{({\bf a}_{\mu})^{d_{\lambda\mu}}})\text{ if } \mu=I(\lambda) \Big\}.$$}
\end{theorem}

\begin{proposition}\label{prop_alg_mor}
The projection map $\pi_{\bf c} \colon Pl(d,n)\to \G_{\bf c}(d,n) $ induces a $R(T_{\bf c})$-algebra homomorphism $$\pi_{\bf c}^*\colon K_{T_{\bf c}}(\G_{\bf c}(d,n))\to K_{T^{n+1}}(Pl(d,n)),$$
where the map $R(T_{\bf c})$ to $R(T^{n+1})$ is given by the inclusion   ${\bf a}_i\to \frac{a_i}{u^{w_i}}$. 
\end{proposition}
\begin{proof}
This follows directly from Theorem \ref{thm_eq_K_thm_pldn} and Theorem \ref{thm_K_T_CGRBDN}. If we substitute ${\bf a}_i$ by $\frac{a_i}{u^{w_i}}$ then $\frac{({\bf a}_{\mu})^{c_\lambda}}{({\bf a}_{\lambda})^{c_\mu}}=\frac{(u{a}_{\mu})^{c_\lambda}}{(u{a}_{\lambda})^{c_\mu}}.$
Thus $$1-\frac{{\bf a}_{\mu}}{({\bf a}_{\lambda})^{d_{\mu\lambda}}}= 1-\frac{u{a}_{\mu}}{(u{a}_{\lambda})^{d_{\mu\lambda}}}=1-{u{a}_{\mu}}\in \frac{\ZZ[a_1^{\pm1}, \dots, a_n^{\pm1},u^{\pm1}]}{(1-ua_{\lambda})}.$$
Therefore, if $x\in \ZZ[{\bf{a}}_1^{\pm1}, \dots, {\bf a}_n^{\pm1}]$ is divisible by $1-\frac{{\bf a}_{\mu}}{({\bf a}_{\lambda})^{d_{\mu\lambda}}}$, then substituting ${\bf a}_i$ by $\frac{a_i}{u^{w_i}}$ in $x$, and viewing the result in the quotient ring $\frac{\ZZ[a_1^{\pm1}, \dots, a_n^{\pm1},u^{\pm1}]}{(1-ua_{\lambda})}$ produces an element divisible by $1-{u{a}_{\mu}}$.
\end{proof}

% {\color{red}
%  Also $R(T^n)=\ZZ[a_1^\pm,\dots,a_n^\pm]$ is a submodule of $R(T^{n+1})=\ZZ[a_1^\pm,\dots,a_n^\pm,u^\pm]$. For each $\lambda\in \mathcal{P}(d,n)$ $R(T^n)$ is same as $\frac{\ZZ[a_1^{\pm}, \dots, a_n^{\pm},u^{\pm}]}{(1-u{\cdot}a_{\lambda})}$ if we put $u=\frac{1}{a_\lambda}$. Also, for each $\lambda$ there is an morphism between $R(T_{\bf c})$ and $R(T^n)$ through the map ${\bf{a}}_i$ to $a_i(a_\lambda)^{w_i}$. Thus ${\bf a}_\lambda$ maps to $a_\lambda(a_\lambda)^{c_\lambda-1}=a_\lambda^{c_\lambda}.$
% Thus there is an onto map from $R(T^n)$ to $R(T_{\bf c})$ given by $a_i^{c_\lambda}$ maps to $\frac{{\bf a}_i^{c_\lambda}}{{\bf a}_\lambda^{w_i}}$ 

%
%For an element $x_\lambda\in \ZZ[{\bf{a}}_1^{\pm}, \dots, {\bf a}_n^{\pm}]$ define $\bar{x}_\lambda\in\frac{\ZZ[a_1^{\pm}, \dots, a_n^{\pm},u^{\pm}]}{(1-u{\cdot}a_{\lambda})}$, be the corresponding element in the quotient space. This gives a $R(T_{\bf c})$-algebra morphism $\phi:K_{T_{\bf c}}(\G_{{\bf c}}(d,n))\to K_{T^{n+1}}(Pl(d,n)) $.For an element $f=(f_\lambda)\in K_{T_{\bf c}}(\G_{\bf c}(d,n))$ define $\phi({f})=(\bar{f}_\lambda)\in K_{T^{n+1}}(Pl(d,n))$.

%Thus the map $\pi_{\bf c}^*$ in Proposition \ref{prop_alg_mor} is actually given by the map ${\bf a}_i$ to $\frac{a_i}{u^{w_i}}$. More explicitly, if we substitute ${\bf a}_i$ by $\frac{a_i}{u^{w_i}}$ in an element $x\in K_{T_{\bf c}}(\G_{\bf c}(d,n))$ we get the element $\pi_{\bf c}^*(x)\in K_{T^{n+1}}(Pl(d,n))$.

\begin{remark}\label{rem_mu_com}
The map $\pi _{\mathbf{c}}^*$ in Proposition \ref{prop_alg_mor}, can be described as a component wise map, with the component indexed by $\lambda \in \mathcal{P}(d,n)$ given by $\mathbf{a_{\mathnormal{i}}}\mapsto a_ia_{\lambda }^{w_i}$.
%  The map $\pi_{\bf c}^*$ in Proposition \ref{prop_alg_mor} can be expressed as component wise, each component corresponds to $\lambda\in \mathcal{P}(d,n)$. The $\lambda$-th component is given by the map ${\bf a}_i$ to ${a_i}{a_\lambda^{w_i}}$.  
\end{remark}

We recall the Schubert class $pS_\lambda\in K_{T^{n+1}}(Pl(d,n))$ from \eqref{eq_pl_sc_bas}. Define 
$$\mathscr{R}:=\pi_{\bf c}^*\Big(K_{T_{\bf c}}(\G_{\bf c}(d,n))\Big)\subset K_{T^{n+1}}(Pl(d,n)).$$ 
  
\begin{lemma}\label{lem_subset}
    $pS_{\lambda}\in \mathscr{R}$.
\end{lemma}
\begin{proof}
Using Remark \ref{rem_pol_2}, $ pS_{\lambda}|_\mu=f_{\lambda\mu}((ua_\nu)_{\nu\leq \mu})$, for some polynomial $f_{\lambda\mu}$. For every $\lambda\in \mathcal{P}(d,n)$ define an element $$x_\lambda\in \bigoplus_{\lambda\in \mathcal{P}(d,n)} R(T_{\bf c})$$ by 
 $x_{\lambda}|_\mu:=f_{\lambda\mu}\Big((\frac{{\bf a}_{\nu}}{({\bf a}_{\mu})^{d_{\nu\mu}}})_{\nu\leq \mu}\Big)\in R(T_{\bf c})$.  Note that $$\frac{({\bf a}_{\nu})}{({\bf a}_{\mu})^{d_{\nu\mu}}}=\frac{ua_\nu}{(ua_\mu)^{d_{\nu\mu}}}=ua_\nu\in \frac{\ZZ[a_1^{\pm1}, \dots, a_n^{\pm1},u^{\pm1}]}{(1-ua_{\mu})}.$$

 Thus $\pi_{\bf c}^*(x_\lambda|_\mu)=pS_\lambda|_\mu$. Since $pS_\lambda\in K_{T^{n+1}}(Pl(d,n))$, it satisfies the GKM condition. Then $x_{\lambda}$ also satisfies the GKM condition and $x_{\lambda}\in K_{T_{\bf c}}(\G_{\bf c}(d,n))$. Hence $pS_\lambda=\pi_{\bf c}^*(x_\lambda)\in \pi_{\bf c}^*(K_{T_{\bf c}}(\G_{\bf c}(d,n)))$.
\end{proof}

For every $\lambda\in \mathcal{P}(d,n)$, we define ${\bf c}S_\lambda|_\mu:=f_{\lambda\mu}\Big((\frac{{\bf a}_{\nu}}{({\bf a}_{\mu})^{d_{\nu\mu}}})_{\nu\leq \mu}\Big)$. In other words,
$${\bf c}S_{\lambda}=(\pi_{\bf c}^*)^{-1}\circ\pi^*(S_{\lambda}).$$
Then using Lemma \ref{lem_subset}, ${\bf c}S_{\lambda}\in K_{T_{\bf c}}(\G_{\bf c}(d,n))$.

\begin{lemma}\label{lem_sc_bas_wt_case}
For every $\lambda\in \mathcal{P}(d,n)$, the element ${\bf c}S_{\lambda}$ satisfies the following condition 
\begin{enumerate}
    \item ${\bf c}S_{\lambda}|_\mu= 0~\text{if }~ \mu \nsucceq \lambda$\\
    \item ${\bf c}S_{\lambda}|_\lambda=\prod_{\nu \in R(\lambda)}(1-\frac{{\bf a}_\nu}{({\bf a}_{\lambda})^{d_{\nu\lambda}}})$
\end{enumerate}
% $${\bf c}S_{\lambda}|_\mu= 
% \begin{cases}
%     \prod_{\nu \in I(\mu)}(1-\frac{{\bf a}_\nu}{({\bf a}_{\mu})^{d_{\nu\mu}}}) &\text{if}~~~ \lambda=\mu\\
%    ~~~~~~~~0~~~~&\text{if }~~~~~ \mu \nsucceq \lambda
% \end{cases}
%   $$  
\end{lemma}
\begin{proof}
The proof follows from Lemma \ref{lem_pslam} and the definition of ${\bf c}S_{\lambda}$.
\end{proof}
%This can also be proved using the character of the torus action.
\begin{proposition}\label{lem_sc_bas_wt_case_gen}
   $\{{\bf c}S_\lambda\}_{\lambda\in \mathcal{P}(d,n)}$ forms a $R(T_{\bf c})$ module basis of $K_{T_{\bf c}}(\G_{\bf c}(d,n))$.   
\end{proposition}
\begin{proof}
    This follows from \cite[Proposition 4.1]{HHH}.
\end{proof}

\begin{lemma}\label{lem_cS_1}
    ${\bf c}S_{(1)}|_{\lambda}=1-\frac{{\bf a}_{(0)}}{({\bf a}_{\lambda})^{d_\lambda}}$.
\end{lemma}
\begin{proof}
 This follows from Lemma \ref{lem_div class} and the definition of ${\bf c}S_{\lambda}$.
\end{proof}

Thus through the map $\pi_{\bf c}^*$ we have a $R(T_{\bf c})$-algebra isomorphism between $K_{T_{\bf c}}(\G_{\bf c}(d,n))$ and $\mathscr{R}\subset K_{T^{n+1}}(Pl(d,n))$, such that ${\bf c}S_\lambda$ maps to $pS_\lambda$.

Let $\varepsilon_{\bf c}: \ZZ[\mathbf{a}_1^{\pm1},\dots,{\bf a}_n^{\pm1}]\to \ZZ$ be the ring homomorphism given by $\varepsilon_{\bf c}({\bf a}_i)=1$. The ordinary $K$-theory $K(\G_{\bf c}(d,n))$ is a ring over $\mathbb{Z}$, can be recovered from the equivariant $K$-theory by base change:
\begin{equation}\label{eq_base_cng}
    K(\mathrm{Gr}_{\bf c}(d,n)) \cong K_{T_{\bf c}}(\mathrm{Gr}_{\bf c}(d,n)) \otimes_{R(T_{\bf c})} \mathbb{Z},
\end{equation}
where $\ZZ$ is considered as a $R(T_{\bf c })$-algebra by the map $\varepsilon_{\bf c}$, see \cite[Proposition 3.25]{KK}. The map $\varepsilon_{\bf c}$ corresponds to trivializing the torus action.
Thus, we have the forgetful map 
$$ K_{T_{\bf c}}(\mathrm{Gr}_{\bf c}(d,n)) \longrightarrow K(\mathrm{Gr}_{\bf c}(d,n)).$$
Let ${\bf c}\mathbb{S}_\lambda$ be the image of ${\bf c}S_\lambda$ through the forgetful map.

\begin{remark}
 A divisive weighted Grassmann orbifolds  $\G_{\bf c}(d,n)$ always has a CW complex structure $\{E(\lambda)\}_{\lambda\in \mathcal{P}(d,n)}$ as in Theorem \ref{thm_CW_cplx_str}. Geometrically, ${\bf c}{{S}}_{\lambda}$ represents the Schubert class in the equivariant $K$-theory of the divisive weighted Grassmann orbifold corresponding to the closure of the cell ${E}(\lambda)$. Thus ${\bf c}{\mathbb{S}}_{\lambda}$ represents the Schubert class in ordinary $K$-theory corresponding to the closure of the cell ${E}(\lambda)$.
\end{remark}

% We have the following product formulae in the equivariant $K$-theory of the divisive weighted Grassmann orbifold.
% \begin{equation}
% 	{\bf{c}}{{S}}_{\lambda}{\bf{c}}{{S}}_{\mu}=\sum_{\nu }{\bf{c}}{K}_{\lambda\mu}^{\nu}{\bf{c}}{{S}}_{\nu}.
% \end{equation}

\begin{proposition}
    The structure constants ${\bf c}K_{\lambda \mu}^{\nu}$ with respect to Schubert basis $\{{\bf c}S_{\lambda}:\lambda\in \mathcal{P}(d,n)\}$ have the following properties.
	\begin{enumerate}
		\item ${\bf{c}}{K}_{\lambda \mu}^{\nu}=0$ unless  $\nu \succeq \lambda, \mu$.   
		\item If $\lambda=\nu$ we have ${\bf{c}}{K}_{\lambda\mu}^{\lambda}={\bf c}{{S}_{\mu}}|_{\lambda}$.
	\end{enumerate}
\end{proposition}
\begin{proof}
    This can be proved using the upper triangularity of ${\bf c}{{S}_{\lambda}}$ as in Lemma \ref{lem_sc_bas_wt_case}.
\end{proof}

We will describe the explicit formulae of ${\bf{c}}{K}_{\lambda \mu}^{\nu}$ in Section \ref{sec_int_coh}.

\section{Twisted factorial Grothendieck polynomials and algebraic localization map}\label{sec_twis_gro_pol}
In this section, we introduce twisted factorial Grothendieck polynomials and prove that twisted factorial Grothendieck polynomials represent the Schubert classes in the equivariant $K$-theory of divisive weighted Grassmann orbifold through an algebraic localization map. We also introduce twisted Grothendieck polynomials and prove that they represent the Schubert structure sheaves in the ordinary $K$-theory of divisive weighted Grassmann orbifolds.

\subsection{Twisted factorial Grothendieck polynomials}
 Let ${\bf c}=(c_\lambda)_{\lambda \in \mathcal{P}(d,n)}$ be a divisive {\Pl} weight vector. By Lemma \ref{lem_div_a,1}, ${\bf c}$ corresponds to $W=(w_1,\dots,w_n)\in (\ZZ_{\geq 0})^n$, and $a=1$.  We assume $w_i=0$ for $i> n$. We define a sub ring $\ZZ[\mathbb{a}_1^{\pm1},\mathbb{a}^{\pm1}_2,\dots]$ of $\ZZ[\xi(x)^{\pm1},a_1^{\pm1},\dots]$ by 
 $$\mathbb{a}_i:=\frac{a_i}{(\xi(x))^{w_i}} \text{ for }i\geq 1.$$ There is a canonical isomorphism of rings $$h: \ZZ[a_1^{\pm1},a_2^{\pm1},\dots][x_1,\dots,x_d]^{S_d} \to \ZZ[\mathbb{a}^{\pm1}_1,\mathbb{a}^{\pm1}_2,\dots][x_1,\dots,x_d]^{S_d}$$ by sending $h(a_i)=(\xi(x))^{w_i}\mathbb{a}_i$ and $h(x_i)=x_i$ for all $i\geq 1$. For every $\lambda\in \mathcal{P}_d$, we define $G_{\lambda}^{\bf c}(x_1,\dots,x_d|\mathbb{a}_1,\mathbb{a}_2,\dots)$ as the image of $G_{\lambda}(x_1,\dots,x_d|1-a_1,1-a_2,\dots)$ under the identification $h$.
 %$ \ZZ[a_1^\pm,\dots,a_n^\pm][x_1,\dots,x_d]^{S_d}\cong\ZZ[\mathbb{a}^{\pm1}_1,\mathbb{a}^{\pm1}_2,\dots,\mathbb{a}^{\pm1}_n][x_1,\dots,x_d]^{S_d}$ defined by $a_i=\mathbb{a}_i^{\bf c}(\pi(x))^{w_i}$. 
 Thus 
 \begin{equation}\label{eq_fac_gro_pol}
G_\lambda^{\bf c}(x_1,\dots,x_d|\mathbb{a}_1,\mathbb{a}_2,\dots):=G_\lambda(x_1,\dots,x_d|1-\mathbb{a}_1{(\xi(x))^{w_1}},1-\mathbb{a}_2{(\xi(x))^{w_2}},\dots).
 \end{equation}

For notational convenience, we denote the polynomial $G_\lambda^{\bf c}(x_1,\dots,x_d|\mathbb{a}_1,\mathbb{a}_2,\dots)$ by $G_\lambda^{\bf c}(x|\mathbb{a})$. %The polynomial $G_\lambda^{\bf c}(x|\mathbb{a})\in \ZZ[\mathbb{a}^{\pm1}_1,\mathbb{a}^{\pm1}_2,\dots][x_1,\dots,x_d]^{S_d}$.
We call the polynomial $G_\lambda^{\bf c}(x|\mathbb{a})$ by `twisted factorial Grothendieck polynomial'.

%It represent the Schubert basis ${\bf c}S_{\lambda}$. 

\begin{example}\label{ex_int_case}
If $\lambda=(1)$ then $G_{(1)}(x|1-a)=1-\xi(x)a_{(0)}$. Now replace $a_i=(\xi(x))^{w_i}\mathbb{a}_i$, then $a_{(0)}=(\xi(x))^{c_{(0)}-1}\mathbb{a}_{(0)}$. Thus $G_{(1)}^{\bf c}(x|\mathbb{a})=1-(\xi(x))^{c_{(0)}}\mathbb{a}_{(0)}.$  
\end{example}

\begin{remark}
 $\{G_{\lambda}^{\bf c}(x_1,\dots,x_d|\mathbb{a}_1,\mathbb{a}_2,\dots)\}_{\lambda \in \mathcal{P}_d}$ form a basis of $\ZZ[\mathbb{a}^{\pm1}_1,\mathbb{a}^{\pm1}_2,\dots][x_1,\dots,x_d]^{S_d}$ as a $\ZZ[\mathbb{a}^{\pm1}_1,\mathbb{a}^{\pm1}_2,\dots]$-algebra.
 \end{remark}

% For ordinary case $G_{\lambda}(x|1-a)$ represent Schubert basis in $K_{T^n}(\G(d,n))$ through the algebraic localization map $\Psi$. Thus, comparing the component we have
% \begin{proposition}
%     $\Psi_{\mu}(G_{\lambda}(x|1-a))=S_{\lambda}|_{\mu}$. In particular, for $\lambda=\mu$    $$\Psi_{\lambda}(G_{\lambda}(x|1-a))=\prod_{\mu\in I(\lambda)}(1-\frac{a_{{\mu}}}{a_{{\lambda}}}).$$
% \end{proposition}

% \begin{example}
%     $\Psi_{\mu}(\xi(x))=\frac{1}{a_{{\mu}}}$. Thus $\Psi_{\mu}(G_{(1)}(x|1-a))=\Psi_{\mu}(1-\xi(x)a_{(0)})=(1-\frac{a_{(0)}}{a_{{\mu}}})$.
% \end{example}
 
\subsection{Algebraic localization map}
Now we construct an algebraic localization map $\Psi^{\bf c}$, an algebra homomorphism from  $\ZZ[\mathbb{a}^{\pm1}_1,\mathbb{a}^{\pm1}_2,\dots][x_1,\dots,x_d]^{S_d}$ to $K_{T_{\bf c}}(\G_{\bf c}(d,n))$ that gives the correspondence between the twisted factorial Grothendieck polynomial $G_{\lambda}^{\bf c}(x|\mathbb{a})$ and the Schubert class ${\bf c}S_\lambda$ in $K_{T_{\bf c}}(\G_{\bf c}(d,n))$ introduced in Section \ref{Sec_Sc_bas_eq_K_th_wgt_gsm}.

%The image of $G_{\lambda}^{\bf c}(x|\mathbb{a})$ under the map $\psi^{\bf c}$ is shown to be a family of Schubert classes.

% Define an algebra homomorphism of $\ZZ[\mathbb{a}_1^{\pm1},\mathbb{a}_2^{\pm1},\dots,\mathbb{a}_n^{\pm1}]$ to $\ZZ[{\mathbf{a}}_1^{\pm}, \dots, {\bf a}_n^{\pm}]$ algebra as  
% $$\psi^{\bf c}_{\mu}: \ZZ[\mathbb{a}^{\pm1}_1,\mathbb{a}^{\pm1}_2,\dots,][x_1,x_2,\dots,x_d]^{S_d}\to \ZZ[{\bf{a}}_1^{\pm}, \dots, {\bf a}_n^{\pm}]$$ by $\psi_\mu^{\bf c}=\psi_\mu\circ h^{-1}$.

% {\color{red}
% For weighted case, $a_i=\mathbb{a}_i\xi(x)^{w_i}$. Thus $${a}_{\widetilde{\mu}_i}=\frac{\mathbb{a}_{\widetilde{\mu}_i}}{(a_{\widetilde{\mu}})^{w_{\widetilde{\mu}_i}}}.$$ Moreover,  ${a}_{{\mu}}^{c_\mu}={\mathbb{a}_{{\mu}}}$.
% the following $\psi_{\mu}^{\bf c}(x_i)=1-\frac{(a_{{\mu}})^{w_{\widetilde{\mu}_i}}}{{\bf{a}}_{\widetilde{\mu}_i}}$, where $({a}_{{\mu}})^{c_\mu}={\bf{a}_{{\mu}}}$. }

\begin{theorem}\label{thm_alg_mor}
There exists a surjective algebra homomorphism $$\Psi^{\bf c}\colon \ZZ[\mathbb{a}^{\pm1}_1,\mathbb{a}^{\pm1}_2,\dots][x_1,\dots,x_d]^{S_d}\to K_{T_{\bf c}}(\G_{\bf c}(d,n))$$ of $\ZZ[\mathbb{a}_1^{\pm1},\mathbb{a}_2^{\pm1},\dots]\to \ZZ[{\mathbf{a}}_1^{\pm1}, \dots, {\bf a}_n^{\pm1}]$ algebra, where $\mathbb{a}_i\to {\bf{a}_i}$ for $1\leq i\leq n$ and $\mathbb{a}_i\to 1$ for $i>n$ such that \[\Psi^{\bf c}(G_{\lambda}^{\bf c}(x|\mathbb{a}))=\begin{cases}
{\bf c}S_{\lambda} &\text{ if } \lambda\in \mathcal{P}(d,n)\\
 0 &\text{ otherwise } \\  
\end{cases}.
\]
\end{theorem}

\begin{proof}
%We construct the algebra homomorphism $\Psi_\mu^{\bf c}$ using $\phi_\mu$ as described in Remark \ref{rmk_phi_mu}, and the isomorphism $h$. 
For every $\mu\in \mathcal{P}(d,n)$, we define an  algebra homomorphism $$\Psi_{\mu}^{\bf c}:\ZZ[\mathbb{a}^{\pm1}_1,\mathbb{a}^{\pm1}_2,\dots][x_1,\dots,x_d]^{S_d}\to \ZZ[\mathbf{a}^{\pm1}_1,\dots,\mathbf{a}^{\pm1}_n] $$ by $\Psi_\mu^{\bf c}(z)=\phi_\mu\circ h^{-1}(z)$. 
  Thus we have the following commutative diagram
\[ \begin{tikzcd}
\ZZ[a_1^{\pm1},a_2^{\pm1},\dots][x_1,\dots,x_d]^{S_d}	 \arrow{r}{\phi_{\mu}}  & {\ZZ[a_1^{\pm1},\dots,a_n^{\pm1},u^{\pm1}]}/{(1-ua_\mu)}  \\%
\ZZ[\mathbb{a}^{\pm1}_1,\mathbb{a}^{\pm1}_2,\dots,][x_1,\dots,x_d]^{S_d}\arrow{u}{\cong} \arrow{r}{\Psi_{\mu}^{\bf c}} & \ZZ[\mathbf{a}^{\pm1}_1,\dots,\mathbf{a}^{\pm1}_n]\arrow{u}{}.
	\end{tikzcd}
	\]

First we show that $\phi_\mu\circ h^{-1}(z)\in \ZZ[\mathbf{a}^{\pm1}_1,\dots,\mathbf{a}^{\pm1}_n]$. For every $\mu \in \mathcal{P}(d,n)$, the ring $\ZZ[\mathbf{a}^{\pm1}_1,\dots,\mathbf{a}^{\pm1}_n]$ can be considered as a sub ring of ${\ZZ[a_1^{\pm1},\dots,a_n^{\pm1},u^{\pm1}]}/{(1-ua_\mu)}$ via the identification $\mathbf{a}_i=a_ia_\mu^{w_i}$, as described in Remark \ref{rem_mu_com}.

 For $1\leq i\leq n, \Psi_\mu^{\bf c}(\mathbb{a}_i)=\phi_\mu(\frac{a_i}{\xi(x)^{w_i}})={a_i}{a_\mu^{w_i}}=\mathbf{a}_i;~ \Psi_\mu^{\bf c}(\mathbb{a}_i)=\phi_\mu(a_i)=1$ for $i>n.$
$$\text{ For } \lambda \in \mathcal{P}(d,n), \Psi_\mu^{\bf c}(G_{\lambda}^{\bf c}(x|\mathbb{a}))=\phi_\mu(G_\lambda(x|1-a))=pS_\lambda|_\mu.$$
% For $i>n$ $\Psi_\mu^{\bf c}(\mathbb{a}_i)=\phi_\mu(\frac{a_i}{\pi(x)^{w_i}})=0$
%Moreover, $\psi_\mu^{\bf c}(G_{\lambda}^{\bf c}(x|\mathbb{a}))=\psi_\mu(G_\lambda(x|1-a))=S_\lambda|_\mu$. 

 Using Lemma \ref{lem_subset}, $pS_\lambda|_\mu\in \ZZ[\mathbf{a}^{\pm1}_1,\dots,\mathbf{a}^{\pm1}_n]$. Thus $\Psi_\mu^{\bf c}$ is well defined.
    Define 
    $$\Psi^{\bf c}\colon \ZZ[\mathbb{a}^{\pm1}_1,\mathbb{a}^{\pm1}_2,\dots,][x_1,\dots,x_d]^{S_d}\to \bigoplus_{\lambda\in \mathcal{P}(d,n)} \ZZ[{\bf{a}}_1^{\pm1}, \dots, {\bf a}_n^{\pm1}]$$ by $(\Psi^{\bf c}(f))_\mu=\Psi^{\bf c}_\mu(f)$. Then $\Psi^{\bf c}(G_{\lambda}^{\bf c}(x|\mathbb{a}))={\bf c}S_{\lambda}$ for $\lambda\in \mathcal{P}(d,n)$ and $\Psi^{\bf c}(G_{\lambda}^{\bf c}(x|\mathbb{a}))=0$ for $\lambda\in \mathcal{P}_d\setminus\mathcal{P}(d,n)$. Therefore, the image of $\Psi^{\bf c}$ is written as $R(T_{\bf c})$ algebra generated by $\{{\bf c}S_{\lambda}\}_{\lambda\in \mathcal{P}(d,n)}$. Consequently, Image$(\Psi^{\bf c})\subset K_{T_{\bf c}}(\G_{\bf c}(d,n)).$ Hence we have the proof.
\end{proof}

% \begin{corollary}
%    There exist surjective algebra morphism $$\Psi^{\bf c}:\ZZ[\mathbb{a}^{\pm1}_1,\mathbb{a}^{\pm1}_2,\dots,\mathbb{a}^{\pm1}_n][x_1,\dots,x_d]^{S_d}\to K_{T_{\bf c}}(\G_{\bf c}(d,n))$$
%  with respect to $\ZZ[\mathbb{a}^{\pm1}_1,\mathbb{a}^{\pm1}_2,\dots,\mathbb{a}^{\pm1}_n]\to \ZZ[{\bf a}^\pm_1,{\bf a}^\pm_2,\dots,{\bf a}^\pm_n]$ which sends $G_{\lambda}(x|\mathbb{a})$ to ${\bf c}S_\lambda$ for $\lambda\in \mathcal{P}_{d,n}$, and  to 0 otherwise.  
% \end{corollary}

% \begin{proof}

% \end{proof}

% \begin{proof}
%    Recall $\phi:\ZZ[{a}^{\pm1}_1,{a}^{\pm1}_2,\dots,{a}^{\pm1}_n][x_1,\dots,x_d]^{S_d}\to K_{T^{n+1}}(Pl(d,n))$ defined by $\pi^*\circ \Psi$. Also,  $\phi(\frac{a_i}{\pi(x)^{w_i}})=\frac{a_i}{u^{w_i}}$. Thus 
% \end{proof}

\begin{corollary}\label{cor_alg_mor} 
For every $\lambda,\mu \in \mathcal{P}(d,n)$, there exist algebra homomorphisms 
$$\Psi^{\bf c}_{\mu}: \ZZ[\mathbb{a}^{\pm1}_1,\mathbb{a}^{\pm1}_2,\dots][x_1,\dots,x_d]^{S_d}\to \ZZ[{\bf{a}}_1^{\pm1}, \dots, {\bf a}_n^{\pm1}]$$ such that the restriction of the Schubert class ${\bf c}S_\lambda\in K_{T^n}(\G_{\bf c}(d,n))$ in the torus fixed point corresponding to $\mu$ can be explicitly computed as the image of the twisted factorial Grothendieck polynomial $G_{\lambda}^{\bf c}(x|\mathbb{a})$. i.e, $\Psi_{\mu}^{\bf c}(G_{\lambda}^{\bf c}(x|\mathbb{a}))={\bf c}S_{\lambda}|_{\mu}$. In particular,
    $$\Psi_{\lambda}^{\bf c}(G_{\lambda}^{\bf c}(x|\mathbb{a}))=\prod_{\nu\in R(\lambda)}(1-\frac{\mathbf{a}_{{\nu}}}{(\mathbf{a}_{{\lambda}})^{d_{\nu \lambda}})}.$$
\end{corollary}

\begin{example}
     $\Psi_\mu^{\bf c}(\xi(x))=\phi_{\mu}(\xi(x))=\frac{1}{a_{{\mu}}}$. Moreover, $\mathbf{a}_{{\mu}}=a_{\mu}^{c_{\mu}}$ in the codomain of $\phi_\mu$. $$\text{ Thus }\Psi_{\mu}^{\bf c}(G_{(1)}^{\bf c}(x|\mathbb{a})=\Psi_{\mu}^{\bf c}(1-(\xi(x))^{c_{(0)}}\mathbb{a}_{{(0)}})=(1-\frac{\mathbf{a}_{{(0)}}}{(\mathbf{a}_{{\mu}})^{d_\mu}}).$$
    Note that $\Psi_{\mu}^{\bf c}(G_{(1)}^{\bf c}(x|\mathbb{a}))={\bf c}S_{(1)}|_{\mu}$, following from Lemma \ref{lem_cS_1}.
\end{example}

\subsection{Twisted Grothendieck polynomials}

\begin{definition}
    For each $\lambda\in \mathcal{P}_d$ , define 
    $$G_{\lambda}^{\bf c}(x):=G_{\lambda}^{\bf c}(x|\mathbb{a})|_{\mathbb{a}=1}.$$
    In other words $G_{\lambda}^{\bf c}(x)=G_{\lambda}(x|1-(\xi(x))^{w_1},1-(\xi(x))^{w_2},\dots).$
\end{definition}

We call the polynomial $G_{\lambda}^{\bf c}(x)$ by twisted  Grothendieck polynomial.

\begin{example}
  From Example \ref{ex_int_case}, we have $G_{(1)}^{\bf c}(x|\mathbb{a})=1-(\xi(x))^{c_0}\mathbb{a}_{(0)}$. Now substituting $\mathbb{a}_i=1$, we have $G_{(1)}^{\bf c}(x)=1-(\xi(x))^{c_0}$.
\end{example}

\begin{proposition}
    $\{G_{\lambda}^{\bf c}(x_1,x_2,\dots,x_d)\}_{\lambda \in \mathcal{P}_d}$ form a $\ZZ$-basis of $\ZZ[x_1,\dots,x_d]^{S_d}$.
\end{proposition}

\begin{theorem}\label{thm_alg_hom_ord_case}There exist a surjective homomorphism $\ZZ[x_1,\dots,x_d]^{S_d}\to K(\G_{\bf c}(d,n))$ as $\ZZ$-algebra,
  which sends $G_{\lambda}^{\bf c}(x_1,x_2,\dots,x_d)$ to ${\bf c}\mathbb{S}_\lambda$ for $\lambda\in \mathcal{P}(d,n)$, and 0 otherwise.  
\end{theorem}
\begin{proof}
    This follows from Theorem \ref{thm_alg_mor} and \eqref{eq_base_cng}.
\end{proof}
\begin{theorem}\label{thm_Z_lin_com}
For every partition $\lambda$, the twisted Grothendieck polynomial $G_\lambda^{\bf c}(x)$ can be expressed as a $\ZZ$-linear combination of the Grothendieck polynomials.
\end{theorem}
\begin{proof}
  Using \cite[Proposition 5.8]{LLS}, factorial Grothendieck polynomials can be expressed in terms of Grothendieck polynomials:  
  \begin{equation}\label{eq_to_ord_ex}
      G_{\lambda}(x|1-a)=G_\lambda(x)+\sum_{\mu}C_{\lambda}^\mu(a)G_\mu(x),
  \end{equation}
Let $J$ be a finite collection of elements in $\{1,2,\dots,n-1\}$ and $a_J=\prod_{j\in J}a_j$. Then 
$C_{\lambda}^\mu(a)=\sum_{J}C(\lambda,\mu,J)a_J$, where  $C(\lambda,\mu,J)\in \ZZ$ is the coefficient of $a_J$ in $C_{\lambda}^\mu(a)$. Define $W_J:=\sum_{j\in J}w_j$.
Substituting $a_i=\xi(x)^{w_i}$ in \eqref{eq_to_ord_ex}, we get 
\begin{align*}
  G_\lambda^{\bf c}(x)
  &=G_\lambda(x)+\sum_\mu  \sum_{J}C(\lambda,\mu,J)\xi(x)^{W_J}G_\mu(x)
  \end{align*}
We substitute $a_i=1$ in \eqref{eq_che_rul}, and iterate this process $k$-times. Then $$(\xi(x))^kG_{\mu}(x)=\sum_{\eta:\mu\xRightarrow[k]{}\eta}(-1)^{|\eta\setminus \mu|}N_{\mu, k}^{\eta}G_{\eta}(x).$$

We denote $\mu\xRightarrow[k]{}\eta$ if
  there is a chain $\mu=\nu^1\Rightarrow \nu^2\Rightarrow \dots\Rightarrow \nu^k\Rightarrow \eta$ and $N_{\mu, k}^{\eta}$ denote the number of possibility of such chains. Thus
 \begin{align*}
    G_\lambda^{\bf c}(x)
  &=G_\lambda(x)+\sum_\mu  \sum_{J}C(\lambda,\mu,J)\sum_{\eta:\mu \xRightarrow[W_J]{} \eta}(-1)^{|\eta \setminus \mu|}N_{\mu, W_J}^{\eta}G_{\eta}(x)\\
  &=G_{\lambda}(x)+\sum_{\eta}\Big(\sum_{\mu}\sum_{J:\mu \xRightarrow[W_J]{} \eta}C(\lambda,\mu,J)(-1)^{|\eta\setminus\mu|}\Big)N_{\mu, W_J}^{\eta}G_\eta(x).
 \end{align*}  
 This completes the proof.
\end{proof}

\begin{corollary}
    ${\bf c}\mathbb{S}_\lambda$ can be written as $\ZZ$-linear combination of $\mathbb{S}_\lambda$.
\end{corollary}

\section{Chevalley rule in the equivariant $K$-theory of divisive weighted Grassmann orbifolds}\label{sec_che_rule}

In this section, we describe the multiplication of any twisted factorial Grothendieck polynomial  $G_{\lambda}^{\bf c}(x|\mathbb{a})$ with $G_{(1)}^{\bf c}(x|\mathbb{a})$. This describes the Chevalley rule in the equivariant and ordinary $K$-theory of divisive weighted Grassmann orbifolds.

\begin{lemma}\label{lem:wgt_ch_rl}
    $$\xi(x)^{c_\lambda}G_{\lambda}^{\bf c}(x|{\mathbb{a}})=\frac{1}{\mathbb{a}_{\lambda}}\sum_{\mu:\lambda\Rightarrow\mu}(-1)^{|\mu\setminus \lambda|}G_{\mu}^{\bf c}(x|{\mathbb{a}}).$$
\end{lemma}

\begin{proof}
    Substituting $a_i=\mathbb{a}_i(\xi(x))^{w_i}$ in \eqref{eq_che_rul}, we get 
    \begin{align*}
      \xi(x)G_{\lambda}^{\bf c}(x|{\mathbb{a}})&=\frac{1}{(\xi(x))^{c_\lambda-1}{\mathbb{a}}_{\lambda}}\sum_{\mu:\lambda\Rightarrow \mu}(-1)^{|\mu\setminus \lambda|}G_{\mu}^{\bf c}(x|{\mathbb{a}})\\
      \implies \xi(x)^{c_\lambda}G_{\lambda}^{\bf c}(x|{\mathbb{a}})&=\frac{1}{\mathbb{a}_{\lambda}}\sum_{\mu:\lambda \Rightarrow\mu}(-1)^{|\mu\setminus \lambda|}G_{\mu}^{\bf c}(x|{\mathbb{a}}).
    \end{align*}
    Hence, we get the proof.
\end{proof}

By applying the Lemma \ref{lem:wgt_ch_rl}, $k$ times in an iterative way we get
\begin{equation}\label{eq_it_pro}
\xi(x)^{k{\cdot}c_\lambda}G_{\lambda}^{\bf c}(x|{\mathbb{a}})=\sum_{\mu:\lambda \xRightarrow[k]{} \mu}L_{\lambda,k}^\mu G_{\mu}^{\bf c}(x|{\mathbb{a}}),
\end{equation}
%where $\lambda \xRightarrow[k]{} \mu$ is defined by the following:

% $\mu$ can be obtained from $\lambda$ after applying $\Rightarrow$ iteratively $k$ times. In other words, 

where we denote $\lambda \xRightarrow[k]{} \mu$ if there exist a chain $\mathscr{C}$ of elements in $\mathcal{P}_d$ as following
\begin{equation}\label{eq_k_Chain}
 \mathscr{C}: \lambda=\nu^1\Rightarrow \nu^2\xRightarrow[d_{\lambda\nu^2}]{} \nu^3\dots\xRightarrow[d_{\lambda\nu^{k-1}}]{} \nu^k\xRightarrow[d_{\lambda\nu^k}]{} \nu^{k+1}=\mu.  
\end{equation}
\begin{enumerate}
\item If $k=1$ then $\lambda \xRightarrow[k]{} \mu$ is equivalent to $\lambda\Rightarrow \mu$.
\item If $c_\lambda=1$ then $\lambda \xRightarrow[k]{} \mu$ means there is a chain $\lambda=\nu^1\Rightarrow \nu^2 \Rightarrow \cdots\Rightarrow \nu^k\Rightarrow \mu$. 
\item If $\mu=\lambda$ or, $|\mu\setminus \lambda|=1$ then for every $k\geq 1$, all possible chains as in \eqref{eq_k_Chain} are explained in Example \ref{eg_mu_eq_lam} and Example \ref{eg_mu_to_lam}. 
\end{enumerate}

For the remaining cases, we can use \eqref{eq_k_Chain} iteratively to reduce each step $\nu^i\xRightarrow[s]{} \nu^{i+1}$ into one of the above three cases. Thus the notation $\lambda \xRightarrow[k]{} \mu$ is well defined. 

Next, we describe the formulae to compute the coefficient $L_{\lambda,k}^\mu$. Using Lemma \ref{lem:wgt_ch_rl},
 if $\lambda\Rightarrow \mu$ then  $$L_{\lambda,1}^\mu=(-1)^{|\mu\setminus \lambda|}\frac{1}{\mathbb{a}_\lambda}.$$

 For $k>1$,
 %let $\lambda=\nu^1\Rightarrow \nu^2 \Rightarrow \cdots\Rightarrow \nu^k\Rightarrow \mu$ be a chain.
 every chain as in \eqref{eq_k_Chain} has a contribution in $L_{\lambda,k}^\mu$ given by $\prod_{i=1}^{k}L_{\nu^i,d_{\lambda\nu^i}}^{\nu^{i+1}}$. Again by using Lemma \ref{lem:wgt_ch_rl}, $L_{\nu^1,1}^{\nu^2}=\frac{(-1)^{|\nu^2\setminus\lambda|}}{\mathbb{a}_{\lambda}}.$ Thus 
\begin{equation}\label{eq_lau_pol}
L_{\lambda,k}^\mu:=\sum_{\mathscr{C}}\frac{(-1)^{|\nu^2\setminus\lambda|}}{\mathbb{a}_{\lambda}}\prod_{i=2}^{k}L_{\nu^i,d_{\lambda\nu^i}}^{\nu^{i+1}},
\end{equation}
where the summation runs over all the chains $\mathscr{C}$ as in \eqref{eq_k_Chain}.

%\substack{\lambda=\nu^1\Rightarrow \nu^2\xRightarrow[d_{\lambda\nu^2}]{} \cdots\\\xRightarrow[d_{\lambda\nu^{k-1}}]{} \nu^k\xRightarrow[d_{\lambda\nu^k}]{}\nu^{k+1}= \mu}

\begin{example}\label{eg_mu_eq_lam}
    If $\mu=\lambda$, then there is only one  chain as in \eqref{eq_k_Chain} given by $\nu^i=\lambda$ for all $i\in \{1,\dots,k+1\}$. 
    Then using the fact that $L_{\lambda,1}^\lambda=\frac{1}{\mathbb{a}_\lambda}$, we get
 \begin{equation}\label{eq_L_lam_k}
     L_{\lambda,k}^{\lambda}=\frac{1}{(\mathbb{a}_{\lambda})^{k}}.
 \end{equation}
\end{example}
\begin{example}\label{eg_mu_to_lam}
    If $\mu$ is obtained by attaching a box to $\lambda$ then
    we have $k$ possibilities of chains as in \eqref{eq_k_Chain}. For every $s\in \{1,2,\dots,k\}$ we get chains as in \eqref{eq_k_Chain} given by
\[\nu^i=\begin{cases}
    \lambda &\text{ for } i\leq s\\
    \mu    &\text{ for } i> s.
\end{cases}
\]
In this case, we have 
\begin{equation}\label{eq_add_one_box}
   L_{\lambda,k}^\mu=-\frac{1}{(\mathbb{a}_{\lambda})^k}\Big(1+\frac{\mathbb{a}_{\lambda}}{(\mathbb{a}_{\mu})^{d_{\lambda\mu}}}+(\frac{\mathbb{a}_{\lambda}}{(\mathbb{a}_{\mu})^{d_{\lambda\mu}}})^2+\dots+(\frac{\mathbb{a}_{\lambda}}{(\mathbb{a}_{\mu})^{d_{\lambda\mu}}})^{k-1}\Big).
\end{equation}
\end{example}

\begin{remark}
    \begin{enumerate}
\item We apply \eqref{eq_lau_pol} iteratively so that, after finitely many steps, each factor of the right hand side of \eqref{eq_lau_pol} reduces to a term of the form $L_{\nu^i,s}^{\nu^{i+1}}$ for some $s\geq 1$ with either $\nu^i=\nu^{i+1}$ or, $\nu^{i+1}$ is obtained by attaching a box to $\nu^i$. Then using \eqref{eq_L_lam_k} and \eqref{eq_add_one_box} we can compute $L_{\lambda,k}^\mu$ explicitly.   

\item If $k=0$, we use the convention $L_{\lambda,k}^\mu=1$ if $\mu=\lambda$ and $L_{\lambda,k}^\mu=0$ if $\mu\neq\lambda$. 
\end{enumerate}
\end{remark}

 $L_{\lambda,k}^\mu\in \ZZ[\mathbb{a}_1^{\pm1},\dots,\mathbb{a}_n^{\pm1}]$ follows from \eqref{eq_it_pro}.  Define $\mathcal{L}_{\lambda,k}^\mu\in \ZZ[\mathbf{a}_1^{\pm1},\dots,\mathbf{a}_n^{\pm1}]$ by substituting $\mathbb{a}_i$ by $\mathbf{a}_i$ in $L_{\lambda,k}^\mu$. We also denote $$\lambda\xRightarrow[k]{}^*\mu\text{ if }\lambda\xRightarrow[k]{}\mu \text{ and }\mu\neq\lambda.$$

\begin{theorem}[Chevalley rule]\label{thm_eq_che}
\begin{equation*}
{\bf c}S_{\lambda}{\bf c}S_{(1)}=(1-\frac{\mathbf{a}_{(0)}}{(\mathbf{a}_{\lambda})^{d_\lambda}}){\bf c}S_\lambda-{\mathbf{a}_{(0)}}\sum_{\substack{\lambda\xRightarrow[d_\lambda]{}^*\mu}}\mathcal{L}_{\lambda, d_{\lambda}}^\mu{\bf c}S_{\mu}.
\end{equation*}
\end{theorem}
\begin{proof}
  We have $c_{(0)}=c_{\lambda}d_{\lambda}$. Substituting $k=d_\lambda$ in \eqref{eq_it_pro}, and then multiplying both side by $\mathbb{a}_{(0)}$ yields the following: $$\mathbb{a}_{(0)}\xi(x)^{c_{(0)}}G_{\lambda}^{\bf c}(x|{\mathbb{a}})=\mathbb{a}_{(0)}\sum_{\mu:\lambda\xRightarrow[d_\lambda]{}\mu}L_{\lambda, d_\lambda}^\mu G_{\mu}^{\bf c}(x|{\mathbb{a}}).$$
 From Example \ref{ex_int_case}, we get $\mathbb{a}_{(0)}\xi(x)^{c_{(0)}}=1-G_{(1)}^{\bf c}(x|{\mathbb{a}})$. Thus 
\begin{equation}\label{eq_pr_Glam_c}
    G_{\lambda}^{\bf c}(x|{\mathbb{a}})G_{(1)}^{\bf c}(x|{\mathbb{a}})=(1-\frac{\mathbb{a}_{(0)}}{(\mathbb{a}_{\lambda})^{d_\lambda}})G_{\lambda}^{\bf c}(x|{\mathbb{a}})-{\mathbb{a}_{(0)}}\sum_{\lambda\xRightarrow[d_\lambda]{}^*\mu}L_{\lambda, d_{\lambda}}^\mu G_{\mu}^{\bf c}(x|{\mathbb{a}}).
\end{equation}
Now using the Theorem \ref{thm_alg_mor}, we can complete the proof.
\end{proof}

\begin{remark}\label{rem_a_0}
 $\mathcal{L}_{\lambda, k}^\mu|_{\mathbf{a}=1}$ is same as $(-1)^{|\mu\setminus \lambda|}N_{\lambda, k}^{\mu}$, where $N_{\lambda, k}^{\mu}\in \ZZ_{\geq 0}$. 
 %denote the number of possibilities of chains $\lambda=\nu^1\Rightarrow \nu^2\Rightarrow \dots\Rightarrow \nu^k\Rightarrow\nu^{k+1}= \mu$. 
 Moreover, using \eqref{eq_lau_pol}, \eqref{eq_L_lam_k} and \eqref{eq_add_one_box} it follows that $(-1)^{|\mu\setminus \lambda|}\mathcal{L}_{\lambda, k}^\mu$ can be factored as $\frac{1}{({\bf a}_\lambda)^{k}}$ multiplied by a polynomial in $\frac{\mathbf{a}_{\nu}}{(\mathbf{a}_{\eta})^{d_{\nu\eta}}}$ for some $\nu$ and $\eta$ with coefficients are non-negative integers. 
\end{remark}

\begin{corollary}[Chevalley rule]\label{thm_wgt_che}
$${\bf c}\mathbb{S}_{\lambda}{\bf c}\mathbb{S}_{(1)}=\sum_{\mu:\lambda\xRightarrow[d_\lambda]{}^*\mu}(-1)^{|\mu\setminus \lambda|-1}N_{\lambda, d_\lambda}^{\mu}{\bf c}\mathbb{S}_{\mu}.$$
\end{corollary}
\begin{proof}
    The proof follows from Theorem \ref{thm_eq_che} and Remark \ref{rem_a_0}.
\end{proof}

If $\mu$ is obtained by attaching one box to $\lambda$ then using \eqref{eq_add_one_box} $\mathcal{L}_{\lambda, d_\lambda}^\mu|_{\mathbf{a}=1}=-d_\lambda$. The coefficients of ${\bf c}\mathbb{S}_{\mu}$ in the right hand side of the expression in Corollary \ref{thm_wgt_che} is $\frac{c_{(0)}}{c_\lambda}$. This coincide with the chevalley formulae in the  cohomology ring of divisive weighted Grassmann orbifold, see \cite[Proposition 6.2]{br}.

\begin{remark}
On the right hand side of \eqref{eq_pr_Glam_c}, the partition $\mu$ may lie in $\mathcal{P}_d\setminus \mathcal{P}(d,n)$. However, in the statement of Theorem \ref{thm_eq_che} and Corollary \ref{thm_wgt_che}, such $\mu$ do not appear. The reason is that, by Theorem \ref{thm_alg_mor}, if $\mu\in \mathcal{P}_d\setminus \mathcal{P}(d,n)$ then $\Psi^{\bf c}(G_{\mu}^{\bf c}(x|{\mathbb{a}}))=0$.
\end{remark}
% \section{equivariant $K$-theory ring of weighted Grassmann orbifold}
% Describe the equivariant $K$-theory ring and its basis follows from the earlier paper 

% I know the following theorem. Let $\mathcal{R}$ be the localization of $\ZZ[\beta][b_1,b_2,\dots]$ by the multiplicative system formed by the products $1+\beta b_i~(i\geq 1)$.

% \begin{theorem}\cite{mc,ikNa}
%     $G_{\lambda}(x_1,x_2,\dots,x_d|b)$ $(\lambda \in \mathcal{P}_d)$ form an $\mathcal{R}$ basis of $\mathcal{R}[x_1,x_2,\dots,x_d]^{S_d}$. 
% \end{theorem}

% Next, we consider $\beta=-1$. Then $\mathcal{R}$ is same as $R(T)$ and we have the following results

% \begin{corollary}
%      $G_{\lambda}(x_1,x_2,\dots,x_n|b)$ $(\lambda \in \mathcal{P}_n)$ form an $R(T)$ basis of $R(T)[x_1,x_2,\dots,x_n]^{S_n}$. 
% \end{corollary}

%  \begin{theorem}
%      There exist a surjective homomorphism of $R(T)$ algebra
%      \[\Psi\colon R(T)[x_1,x_2,\dots,x_d]^{S_d}\to K_{T^n}^0(\G(d,n))\]
%      which sends $G_{\lambda}(x|1-\alpha_1,\dots,1-\alpha_n)$ to $S_{\lambda}$ if $\lambda_1\leq n-d$ and 0 othewise.
%  \end{theorem}

% Look at Theorem 3.4 of Macnamara. 

% \begin{theorem}
%     There exist a subring $G_n$ of $\ZZ[\beta][x_1,\dots,x_n]$ such that $G_{\lambda}(x_1,x_2,\dots,x_n|b)$ $(\lambda \in \mathcal{P}_n)$ form an $R(T)$ basis of $R(T)\otimes G_n$.
% \end{theorem}

%  \begin{theorem}
%      There exist a surjective homomorphism of $R(T)$ algebra
%      \[\Psi\colon R(T)\otimes G_n\to K_T^*(\G(n,n+k))\]
%      which sends $G_{\lambda}(x|1-e^t)$ to $\sigma_{\lambda}^T$ if $\lambda_1\leq k$ and 0 othewise.
%  \end{theorem}

\section{Structure constants of the equivariant $K$-theory of divisive weighted Grassmann orbifolds}\label{sec_int_coh}
In this Section, we explicitly compute the structure constants 
in $K_{T_{\bf c}}(\G_{\bf c}(d,n))$ with respect to the Schubert classes ${\bf c}S_\lambda$. In \cite{PY, PY1} Pechenik-Yong describe a combinatorial formulae  the equivariant structure constants ${{K}}_{\lambda\mu}^{\nu}$ with respect to the Schubert classes $S_{\lambda}$.
\begin{equation}
	{{S}}_{\lambda}{{S}}_{\mu}=\sum_{\nu }{{K}}_{\lambda\mu}^{\nu}{{S}}_{\nu}.
\end{equation}

% \begin{equation}
% 	{{S}}_{\lambda}{{S}}_{\mu}=\sum_{\nu\succeq\lambda,\mu }{K}_{\lambda\mu}^{\nu}{{S}}_{\nu}.
% \end{equation}

\begin{lemma}\label{lem_special_mu}
    $K_{\lambda \mu}^\nu$ can be written as a polynomial in $\frac{a_{i}}{a_{i+1}}$ with $1\le i\leq \tilde\nu_d-1$.
\end{lemma}

\begin{proof}
From Lemma \ref{lem_loc_sc_bas_gsm}, it follows that $S_{\lambda}|_\mu\in \ZZ[a_1^{\pm1},\dots,a_{\tilde{\mu}_d}^{\pm1}]$. Let $\rho$ be a minimal Schubert symbol such that $\lambda\preceq \rho$ and $\mu\preceq \rho$ both hold. Here $\rho$ is minimal means that there is no other Schubert symbol $\rho'$ such that $\rho'\prec\rho$ and $\lambda\preceq \rho'$,  $\mu\preceq \rho'$ both hold. Then, using the upper triangularity of $S_\lambda$, we have  $K_{\lambda\mu}^{\nu}=0$ if $\nu\prec\rho$.

For $\nu=\rho$  we have $K_{\lambda \mu}^\rho S_\rho|_\rho=S_{\lambda}|_\rho S_\mu|_\rho$. 
Thus $K_{\lambda\mu}^{\rho}\in \ZZ[a_1^{\pm1},\dots,a_{\tilde{\rho}_d}^{\pm1}]$. We complete the prove by induction. For any arbitrary $\nu\succ\rho$ we have
$$K_{\lambda \mu}^\nu S_\nu|_\nu=S_{\lambda}|_\nu S_\mu|_\nu-\sum_{\eta:\rho\preceq\eta\prec\nu}K_{\lambda \mu}^\eta S_\eta|_\nu.$$ 
From induction hypothesis, $K_{\lambda \mu}^\eta\in \ZZ[a_1^{\pm1},\dots,a_{\tilde{\eta}_d}^{\pm1}]$ and $S_{\eta}|_\nu\in \ZZ[a_1^{\pm1},\dots,a_{\tilde{\nu}_d}^{\pm1}]$ follows from Lemma \ref{lem_loc_sc_bas_gsm}. Thus $K_{\lambda \mu}^\nu\in \ZZ[a_1^{\pm1},\dots,a_{\tilde{\nu}_d}^{\pm1}]$. Also, using \cite[Corollary 1.5]{PY}, $K_{\lambda \mu}^\nu$ can be written as a polynomial in $\frac{a_{i}}{a_{i+1}}$. Hence we have the proof. 
\end{proof}
Let $I$ be a finite collection of elements in $\{1,2,\dots,\tilde{\nu}_d-1\}$. We define ${u}_i=\frac{{a}_{i}}{{a}_{i+1}}$ and $U_I:=\prod_{i\in I} u_i$. Using \cite[Corollary 1.5]{PY}, ${K}_{\lambda\mu}^{\nu}$ can be written as the following: 
\begin{equation}\label{eq_st_co_gsm}
    {K}_{\lambda\mu}^{\nu}=\sum_{I}C(\lambda,\mu,\nu,I)U_I,
\end{equation}
  $C(\lambda,\mu,\nu,I)\in \ZZ$ is the coefficient of $U_I$. 
 %Thus  $C(\lambda,\mu,\nu,I)=0$ if there is no monomial in the expansion of  ${K}_{\lambda\mu}^{\nu}$ which is constant multiple of $U_I$. 
 Moreover, $(-1)^{|\nu|-|\lambda|-|\mu|}C(\lambda,\mu,\nu,I)\in \ZZ_{\geq 0}$.

\begin{remark}
    The elements in $I$ need not be always distinct; some elements can occur finitely many times. For example, $U_I$ could be $u_3u_4^2$ corresponding to $I=\{3,4,4\}$.
\end{remark}

For each $1< i\leq \tilde{\nu}_d-1$, one can construct two minimal Schubert symbols $\eta$ and $\eta'$ such that $\eta=(i,i+1)\eta'$. To describe this explicitly, let $i_1,i_2,\dots,i_{d-1}$ the smallest $d-1$ elements in $\{1,\dots,d+1\}\setminus\{i,i+1\}$. Define $\eta=(i_1,\dots,i_{d-1},i+1)$ and $\eta'=(i_1,\dots,i_{d-1},i)$. Then $w_{i}-w_{i+1}=c_{\eta'}-c_\eta$. Define $d_i=c_{\eta'}-c_\eta=(d_{\eta'\eta}-1)c_\eta$. Note that both $\eta\prec\nu$ and $\eta'\prec\nu$ hold. Thus $c_\nu$ divides $c_{\eta'}-c_{\eta}$. Define $d_{i,\nu}$ by $c_{\eta'}-c_{\eta}=d_{i,\nu}c_\nu$ and $d_{I,\nu}=\sum_{i\in I}d_{i,\nu}$. Define  $\mathbb{u}_i=\frac{\mathbb{a}_{i}}{\mathbb{a}_{i+1}}$, $\mathbb{U}_I=\prod_{i\in I}\mathbb{u}_i$ and similarly $\mathfrak{u}_i=\frac{\mathbf{a_{i}}}{\mathbf{a}_{i+1}}$, $\mathcal{U}_I=\prod_{i\in I}\mathfrak{u}_i$.

\begin{theorem}\label{str_cst_eq_case}
For $\lambda,\mu,\eta\in \mathcal{P}(d,n)$ with $\eta\succeq \lambda, \mu$, we have
    $${\bf c}{K}_{\lambda\mu}^\eta=\sum_{\nu:\nu \succeq\lambda,\mu  }\sum_{I: \nu \xRightarrow[d_{I,\nu}]{} \eta}C(\lambda,\mu,\nu,I)\mathcal{U}_I \mathcal{L}_{\nu,d_{I,\nu}}^\eta.$$
\end{theorem}
\begin{proof}
Using Theorem \ref{thm_sur_hom} together with \eqref{eq_st_co_gsm} we have
\begin{equation*}
     G_{\lambda}(x|1-a)G_{\mu}(x|1-a)=\sum_{\nu\succeq\lambda,\mu }\sum_{I}C(\lambda,\mu,\nu,I){U}_IG_\nu(x|1-a).
\end{equation*}
    Now substituting $a_i=\mathbb{a}_i(\xi(x))^{w_i}$ and using Lemma \ref{lem_special_mu} we get
    
    \begin{align*}
    G_\lambda^{\bf c}(x|\mathbb{a}) G_\mu^{\bf c}(x|\mathbb{a})&=\sum_{\nu\succeq\lambda,\mu }\sum_{I}C(\lambda,\mu,\nu,I)\mathbb{U}_I\xi(x)^{\sum_{i\in I}(w_{i}-w_{i+1})} G_\nu^{\bf c}(x|\mathbb{a})\\
    &=\sum_{\nu\succeq\lambda,\mu }\sum_{I}C(\lambda,\mu,\nu,I)\mathbb{U}_I\xi(x)^{c_\nu d_{I,\nu}}G_\nu^{\bf c}(x|\mathbb{a})\\
    &=\sum_{\nu\succeq\lambda,\mu }\sum_{I}C(\lambda,\mu,\nu,I)\mathbb{U}_I\Big(\sum_{\eta:\nu \xRightarrow[d_{I,\nu}]{} \eta}L_{\nu,d_{I,\nu}}^\eta G_\eta^{\bf c}(x|\mathbb{a})\Big)\\
    &=\sum_{\eta}\Big{(}\sum_{\nu:\nu \succeq\lambda,\mu }\sum_{I:\nu \xRightarrow[d_{I,\nu}]{} \eta}C(\lambda,\mu,\nu,I)\mathbb{U}_I L_{\nu,d_{I,\nu}}^\eta\Big{)}G_\eta^{\bf c}(x|\mathbb{a}).
\end{align*}
Now using the Theorem \ref{thm_alg_mor}, we can complete the proof.
\end{proof}

\begin{corollary}\label{cor_st_con_ord_case}
 Structure constants ${\bf{c}}{\mathscr{K}}_{\lambda\mu}^{\eta}$ with respect to the Schubert basis $\{{\bf{c}}{\mathbb{S}}_{\lambda}\}_{\lambda\in \mathcal{P}(d,n)}$ in the ordinary $K$-theory are given by 
 $$ {\bf{c}}{\mathscr{K}}_{\lambda\mu}^{\eta}={\mathscr{K}}_{\lambda\mu}^{\eta}+\sum_{\nu:~\eta\succ\nu \succeq\lambda,\mu  }\sum_{I: \nu \xRightarrow[d_{I,\nu}]{} \eta}(-1)^{|\eta\setminus \nu|}C(\lambda,\mu,\nu,I)N_{\nu,d_{I,\nu}}^{\eta}.$$ 
 % $$ {\bf{c}}{\mathscr{K}}_{\lambda\mu}^{\eta}=\sum_{\nu:\nu \succeq\lambda,\mu  }\sum_{I: \nu \xRightarrow[d_{I,\nu}]{} \eta}C(\lambda,\mu,\nu,I)(-1)^{|\eta\setminus \nu|}N_{\nu,d_{I,\nu}}^{\eta}.$$  
\end{corollary}
\begin{proof}
We substitute ${\bf a}_i=1$ in ${\bf c}{K}_{\lambda\mu}^\eta$ and use the fact $\mathcal{L}_{\eta,d_{I,\nu}}^\eta|_{{\bf a}_i=1}=1$. Then the proof follows using Remark \ref{rem_a_0} and Theorem \ref{str_cst_eq_case}.
\end{proof}

% \begin{remark}
% The structure constant $ {\bf{c}}{\mathscr{K}}_{\lambda\mu}^{\eta}$ is non-zero only if $|\eta|\geq|\lambda|+|\mu|$. If $|\eta|=|\lambda|+|\mu|$ the structure constant $ {\bf{c}}{\mathscr{K}}_{\lambda\mu}^{\eta}$ coincide with structure constant in the cohomology ring of divisive weighted Grassmann orbifold.
% \end{remark}

%Next, we prove the positivity of the structure constants ${\bf c}K_{\lambda\mu}^{\eta}$.
For $\lambda\preceq \mu$, we denote  $\mathbf{a}_{\lambda\mu}:=\frac{\mathbf{a}_\lambda}{(\mathbf{a}_\mu)^{d_{\lambda\mu}}}.$

\begin{theorem}\label{thm_postivity}
    $(-1)^{|\eta|-|\lambda|-|\mu|}\,{\bf c}K_{\lambda\mu}^{\eta}$ is expressible as a Laurent polynomial with nonnegative integer coefficients in the $\mathbf{a}_{\lambda\mu}$. 
\end{theorem}
\begin{proof}
We have $w_{i}-w_{i+1}=c_{\eta}-c_{\eta'}=d_{i,\nu}c_\nu$ and $d_{i,\nu}=d_{\eta\nu}-d_{\eta'\nu}$. Thus $\mathfrak{u}_i=\frac{{\bf a}_{i}}{{\bf a}_{i+1}}=\frac{\bf a_{\eta}}{\bf a_{\eta'}}$. Using Remark \ref{rem_a_0}, $(-1)^{|\eta|-|\nu|}\mathcal{L}_{\nu,d_{I,\nu}}^\eta$ can be written as $\frac{1}{({\bf a}_\nu)^{d_{I,\nu}}}$ times a polynomial in ${\bf a}_{\nu\mu}$ with coefficients non-negative integers. Moreover, $$\mathfrak{u}_i\frac{1}{({\bf a}_\nu)^{d_{i,\nu}}}=\frac{\bf a_{\eta}}{({\bf a}_\nu)^{d_{\eta\nu}}}\frac{({\bf a}_\nu)^{d_{\eta'\nu}}}{\bf a_{\eta'}}=\frac{\mathbf{a}_{\eta\nu}}{\mathbf{a}_{\eta'\nu}};~~~~ \mathcal{U}_I\frac{1}{({\bf a}_\nu)^{d_{I,\nu}}}=\prod_{i\in I} \mathfrak{u}_i\frac{1}{({\bf a}_\nu)^{d_{i,\nu}}}.$$
Now the proof follows from Theorem \ref{str_cst_eq_case}, using 
$(-1)^{|\nu|-|\lambda|-|\mu|}C(\lambda,\mu,\nu,I)\in \ZZ_{\geq 0}$. 
\end{proof}

\begin{corollary}
$(-1)^{|\eta|-|\lambda|-|\mu|}{\bf c}\mathscr{K}_{\lambda\mu}^{\eta}\in \ZZ_{\geq 0}$ follows from Theorem \ref{thm_postivity}.    
\end{corollary}

\begin{example}\label{eg_wgt_pro_sp}
  Consider the divisive weighted projective space $\WW P(c_1,\dots,c_n)$ as defined in \cite{HHRW}. $\WW P(c_1,\dots,c_n)$ is same as $\WG(1,n)$ for $W=(c_1-1,\dots,c_n-1)$ and $a=1$. We have a natural $T^{n}$-action on $\WW P(c_1,\dots,c_n)$ by coordinate wise multiplications. The Schubert symbols are given by $I(1,n)=\{(1),(2),\dots,(n)\}$. 
  
  % Note that the partitions are given by $\mathcal{P}(1,n)=\{(0),(1),\dots,(n-1)\}$. 
   %For every $(i)\in \mathcal{P}(1,n)$ the factorial Grothendieck polynomial $G_{(i)}(x|b)$ is given by $$G_{(0)}(x|b)=1 \text{ and } G_{(i)}(x|b)=(x\oplus b_1)\dots (x\oplus b_{i}) \text{ for } 1\le i\le n-1.$$

For $(i)\in I(1,n)$, the Schubert class $S_{(i)}\in K_{T^n}(\G(1,n))$ is given by $$S_{(1)}|_{(j)}=1.~\text{ For } i\geq 2, S_{(i)}|_{(j)}=\prod_{k=1}^{i-1}(1-\frac{{a}_k}{{a}_j}).$$

  For every $i\in \{1,2,\dots,n\}$, the Schubert class ${\bf c}S_{(i)}\in K_{T^n}(\WG(1,n))$ is given by   $${\bf c}S_{(1)}|_{(j)}=0.~ ~\text{ For } i\geq 2 \text{ and } j\geq i,  {\bf c}S_{(i)}|_{(j)}=\prod_{k=1}^{i-1}(1-\frac{\mathbf{a}_k}{\mathbf{a}_j^{d_{k,j}}}),$$
  % $${\bf c}S_{(2)}|_{(j)}=(1-\frac{\mathbf{a}_1}{\mathbf{a}_j^{d_{j}}}).$$
where  $d_{i,j}=\frac{c_i}{c_j}$, and $d_j=\frac{c_1}{c_j}$. The weighted Chevalley rule is given by
  \begin{equation*}
{\bf c}S_{(i)}{\bf c}S_{(1)}=(1-\frac{\mathbf{a}_{1}}{(\mathbf{a}_{i})^{d_{i}}}){\bf c}S_{(i)}-{\mathbf{a}_{1}}\sum_{j\geq 1}\mathcal{L}_{(i), d_{i}}^{(i+j)}{\bf c}S_{(i+j)}.
\end{equation*}

Using Example \ref{eg_mu_to_lam}, $\mathcal{L}_{(i), k}^{(i+1)}=-\frac{1}{(\mathbf{a}_i)^{k}}\Big(1+\frac{\mathbf{a}_i}{(\mathbf{a}_{i+1})^{d_{i,i+1}}}+(\frac{\mathbf{a}_i}{(\mathbf{a}_{i+1})^{d_{i,i+1}}})^2+\dots+(\frac{\mathbf{a}_{i}}{(\mathbf{a}_{i+1})^{d_{i,i+1}}})^{k-1}\Big)$. 

 For $k\geq 2$, consider $(i)\xRightarrow[k]{}(i+2)$. In this case, we have $\frac{k(k-1)}{2}$ many possibilities of chains as in \eqref{eq_k_Chain} containing $k_1+1$ many $(i)$, $k_2+1$ many $(i+1)$ and $k_3+1$ many $(i+2)$, where $k_1,k_2,k_3\geq 0$ such that $k_1+k_2+k_3=k-2$. Thus  
 $$\mathcal{L}_{(i),k}^{(i+2)}=\mathcal{L}_{(i+1), d_{i,i+1 }}^{(i+2)}\sum_{\substack{k_1,k_2,k_3\geq 0\\k_1+k_2+k_3=k-2}}\frac{1}{({\bf a}_i)^{k_1+1}({\bf a}_{i+1})^{d_{i,i+1}k_2}({\bf a}_{i+2})^{d_{i,i+2}k_3}}.$$
Moreover, $\mathcal{L}_{(i), d_i}^{(i+3)},\dots , \mathcal{L}_{(i), d_i}^{(i+j)}$ can be computed iteratively using \eqref{eq_lau_pol}. Thus $\mathcal{L}_{(i), k}^{(i+1)}|_{{\bf a}_\ell=0}=k$, 
 $\mathcal{L}_{(i), k}^{(i+2)}|_{{\bf a}_\ell=0}=\frac{d_{i,i+1}k(k-1)}{2}$ and so on. In this way, one can compute  the structure constants of $K_{T^n}(\WW P(c_1,\dots,c_n))$ using Theorem \ref{str_cst_eq_case}. 
%   \begin{equation*}
% {\bf c}\mathbb{S}_{(i)}{\bf c}\mathbb{S}_{(1)}=(-1)^{j-1}\sum_{j=1}^{d_{i}}{d_i\choose j}{\bf c}\mathbb{S}_{(i+j)}.
% \end{equation*}
% Moreover, if $S_iS_j=\sum_{k:k\ge i,j}(\sum_IC(i,j,k,I)U_I)S_k$ then
%  $$ {\bf{c}}\mathbb{S}_{i}{\bf c}\mathbb{S}_{j}=\sum_{\ell:\ell \geq i,j  }\sum_{s=1}^{d_{I,\ell}}C(i,j,\ell,I)(-1)^{s}{d_{I,\ell}\choose s}{\bf c}\mathbb{S}_{\ell+s},$$
%  where $I$  be a finite collection of elements in $\{1,2,\dots,n-1\}$.
\end{example}

In the following example, we compute the structure constants in the equivariant $K$-theory of a divisive weighted Grassmann orbifold $\G_{\bf{c}}(2,4)$. 

\begin{example}\label{eg_wgt_gr_24}
Consider a Pl\"{u}cker weight vector ${\bf{c}}:=(\alpha\beta,\alpha\beta,\alpha\beta,\alpha,\alpha,\alpha)$, where $\alpha$ and $\beta$ are two positive integers. Then ${\bf{c}}$ is a Pl\"{u}cker weight following Example \ref{ex_pl_wgt_vec} and it is divisive. $\G_{\bf c}(2,4)$ is homeomorphic to $\G_{\bf c'}(2,4)$, where ${\bf c'}=(\beta,\beta,\beta,1,1,1)$ which correspond to $W=(\beta-1,0,0,0)$ and $a=1$ as in Lemma \ref{lem_div_a,1}. In the equivariant $K$-theory ring $K_{T^4}(\G(2,4))$ we have the following:
\begin{align*}
{S}_{(2,0)}{S}_{(2,1)}={S}_{(2,0)}|_{(2,1)}{S}_{(2,1)}+{K}_{(2,0),(2,1)}^{(2,2)}{S}_{(2,2)}.
%\Big{(}(1-\frac{a_1}{a_2})\frac{a_3}{a_4}+(1-\frac{a_2}{a_3})\frac{a_3}{a_4}-(1-\frac{a_1}{a_2})(1-\frac{a_2}{a_3})\frac{a_3}{a_4}\\+(1-\frac{a_3}{a_4})\frac{a_2}{a_4}-(1-\frac{a_1}{a_2})(1-\frac{a_3}{a_4})\frac{a_2}{a_4}\Big{)}{S}_{(2,2)}.
\end{align*}

$${K}_{(2,0),(2,1)}^{(2,1)}={S}_{(2,0)}|_{(2,1)}=(1-\frac{a_3}{a_4})(1-\frac{a_1}{a_4})=(1-u_3)(1-u_1u_2u_3)=1-u_3-u_1u_2u_3+u_1u_2u_3^2.$$

\[C((2,0),(2,1),(2,1);I)=\begin{cases}
1&\text{ if } I=\emptyset\\
 -1 &\text{ if } I=\{3\}\\
   -1 &\text{ if } I=\{1,2,3\}\\
   1 &\text{ if } I=\{1,2,3,3\}.
\end{cases}\]
We calculate the value of $d_{I,\nu}$ for $\nu=(2,1)$. If $I=\emptyset$ and $I=\{3\}$, then $d_{I,\nu}=0$. If $I=\{1,2,3\}$ and $I=\{1,2,3,3\}$, then $d_{I,\nu}=\beta-1$. 

Moreover, $\mathcal{L}_{(2,1),0}^{(2,1)}=0$ and $\mathcal{L}_{(2,1),\beta-1}^{(2,1)}=\frac{1}{(\mathbf{a}_{(2,1)})^{\beta-1}}=\frac{1}{(\mathbf{a}_2\mathbf{a}_4)^{\beta-1}}$. Thus

\begin{align*}
{\bf c}{K}_{(2,0),(2,1)}^{(2,1)}&=1-{\mathfrak{u}_3}-\frac{\mathfrak{u}_1\mathfrak{u}_2\mathfrak{u}_3}{(\mathbf{a}_2\mathbf{a}_4)^{\beta-1}}+\frac{\mathfrak{u}_1\mathfrak{u}_2\mathfrak{u}_3^2}{(\mathbf{a}_2\mathbf{a}_4)^{\beta-1}}\\
    &=1-\frac{\mathbf{a}_2\mathbf{a}_3}{\mathbf{a}_2\mathbf{a}_4}-\frac{\mathbf{a}_1\mathbf{a}_2}{(\mathbf{a}_2\mathbf{a}_4)^\beta}+\frac{\mathbf{a}_1\mathbf{a}_2\mathbf{a}_2\mathbf{a}_3}{(\mathbf{a}_2\mathbf{a}_4)^{\beta+1}}\\
    &=(1-\frac{\mathbf{a}_2\mathbf{a}_3}{\mathbf{a}_2\mathbf{a}_4})(1-\frac{\mathbf{a}_1\mathbf{a}_2}{(\mathbf{a}_2\mathbf{a}_4)^\beta})={\bf c}{S}_{(2,0)}|_{(2,1)}.
\end{align*}
The formulae of $ {K}_{(2,0),(2,1)}^{(2,2)}$ is described in \cite[Example 1.4]{PY1}.
{\small \begin{align*}
 {K}_{(2,0),(2,1)}^{(2,2)}&= (1-\frac{a_1}{a_2})\frac{a_3}{a_4}+(1-\frac{a_2}{a_3})\frac{a_3}{a_4}-(1-\frac{a_1}{a_2})(1-\frac{a_2}{a_3})\frac{a_3}{a_4}+(1-\frac{a_3}{a_4})\frac{a_2}{a_4}-(1-\frac{a_1}{a_2})(1-\frac{a_3}{a_4})\frac{a_2}{a_4} \\
 &=(1-u_1)u_3+(1-u_2)u_3-(1-u_1)(1-u_2)u_3+(1-u_3)u_2u_3-(1-u_1)(1-u_3)u_2u_3\\
 &=u_3-u_1u_2u_3+(1-u_3)u_1u_2u_3=u_3-u_1u_2u_3^2.
\end{align*}}

\[C((2,0),(2,1),(2,2);I)=\begin{cases}
   -1 &\text{ if } I=\{1,2,3,3\}\\
   1 &\text{ if } I=\{3\}.
\end{cases}\]

If $\nu=(2,1)$, $\mathcal{L}_{(2,1),0}^{(2,2)}=0$ and $\mathcal{L}_{(2,1),\beta-1}^{(2,2)}=\frac{-1}{({\bf a}_2{\bf a}_4)^{\beta-1}}(1+\frac{\mathbf{a}_2\mathbf{a}_4}{\mathbf{a}_3\mathbf{a}_4}+(\frac{\mathbf{a}_2\mathbf{a}_4}{\mathbf{a}_3\mathbf{a}_4})^2+\dots+(\frac{\mathbf{a}_2\mathbf{a}_4}{\mathbf{a}_3\mathbf{a}_4})^{\beta-2})$.

If $\nu=(2,2)$, $\mathcal{L}_{(2,2),0}^{(2,2)}=1$ and $\mathcal{L}_{(2,2),\beta-1}^{(2,2)}=\frac{1}{({\bf a}_3{\bf a}_4)^{\beta-1}}$. Thus
{\small \begin{align*}
{\bf c}{K}_{(2,0),(2,1)}^{(2,2)}&={\mathfrak{u}_3}-\frac{\mathfrak{u}_1\mathfrak{u}_2\mathfrak{u}_3^2}{(\mathbf{a}_3\mathbf{a}_4)^{\beta-1}}+\frac{(\mathfrak{u}_1\mathfrak{u}_2\mathfrak{u}_3-\mathfrak{u}_1\mathfrak{u}_2\mathfrak{u}_3^2)}{(\mathbf{a}_2\mathbf{a}_4)^{\beta-1}}(1+\frac{\mathbf{a}_2\mathbf{a}_4}{\mathbf{a}_3\mathbf{a}_4}+(\frac{\mathbf{a}_2\mathbf{a}_4}{\mathbf{a}_3\mathbf{a}_4})^2+\dots+(\frac{\mathbf{a}_2\mathbf{a}_4}{\mathbf{a}_3\mathbf{a}_4})^{\beta-2}).
\end{align*}}
\end{example}

{\bf Acknowledgement.} 
The author would like to thank Takeshi Ikeda, Soumen Sarkar, and Parameshwaran Sankaran for many valuable discussions. The author thanks Chennai Mathematical Institute and Infosys for the postdoctoral fellowship. The author thanks Waseda University and the JSPS Postdoctoral Research Fellowship. This work was supported by JSPS KAKENHI Grant Numbers 24KF0258, 25KF0074.

\bibliographystyle{abbrv}
\bibliography{ref-Wedge1}

\end{document}